\documentclass[11pt]{article}

\usepackage{amsmath}
\usepackage{amsfonts}
\usepackage{amsthm}
\usepackage{amssymb}
\usepackage{mathrsfs}
\usepackage{verbatim}
\usepackage{enumerate}
\usepackage{graphicx}
\usepackage{longtable}
\usepackage[all,cmtip]{xy}
\usepackage{upref}
\usepackage[hang,small,bf]{caption}
\usepackage{marginnote}
\usepackage{a4wide}
\usepackage{mathtools,tensor}
\usepackage{xfrac}
\usepackage{color}
\usepackage{bm}
\usepackage{hyperref}
\usepackage{nomencl}
\usepackage{xtab,cuted}

\newtheorem{thm}{Theorem}[section]
\newtheorem{prp}[thm]{Proposition}
\newtheorem{lem}[thm]{Lemma}
\newtheorem{cor}[thm]{Corollary}
\newtheorem{con}[thm]{Question}
\newtheorem*{con*}{Question}
\theoremstyle{definition}
\newtheorem{dfn}[thm]{Definition}
\newtheorem{rmk}[thm]{Remark}

\numberwithin{equation}{section}

\newcommand{\Z}{{\mathbb Z}}

\newcommand{\R}{{\mathbb R}}
\newcommand{\N}{{\mathbb N}}
\newcommand{\Y}{{\mathbb Y}}
\newcommand{\A}{{\mathbb A}}
\newcommand{\C}{{\mathbb C}}

\newcommand{\df}{{\mathfrak d}}
\newcommand{\Df}{{\mathfrak D}}

\newcommand{\op}[1]{\operatorname{#1}}

\newcommand{\NN}{{\mathcal N}}
\newcommand{\OO}{{\mathcal O}}
\newcommand{\WW}{{\mathcal W}}
\newcommand{\ZZ}{{\mathcal Z}}

\newcommand{\dist}{{\mathrm{dist}}}
\newcommand{\st}{{\mathrm{st}}}
\newcommand{\rep}{{\mathrm{rep}}}
\newcommand{\om}{\mathrm{\omega}}

\newcommand{\Vol}{\mathrm{Vol}}
\newcommand{\Fix}{\mathrm{Fix\,}}

\newcommand{\Crit}{\mathrm{Crit\,}}
\newcommand{\CAL}{\mathrm{CAL}}
\newcommand{\id}{\mathrm{id}}

\newcommand{\di}{{\mathrm d}}

\newcommand{\ta}{{\mathrm T}}
\newcommand{\dR}{{\mathrm{dR}}}

\newcommand{\p}{\partial}
\newcommand{\into}{\hookrightarrow}
\newcommand{\x}{\times}
\newcommand{\wh}{\widehat}

\newcommand{\beq}{\begin{equation}}
\newcommand{\beqn}{\begin{equation}\nonumber}
\newcommand{\eeq}{\end{equation}}
\newcommand{\bea}{\begin{equation}\begin{aligned}}
\newcommand{\bean}{\begin{equation}\begin{aligned}\nonumber}
\newcommand{\eea}{\end{aligned}\end{equation}}

\AtEndDocument{\bigskip{\footnotesize
  \textsc{Ruprecht-Karls-Universit\"at Heidelberg, Mathematisches Institut, Im Neuenheimer Feld 205, 69120 Heidelberg, Germany} \par  
  \textit{E-mail address}: \texttt{\href{mailto:gbenedetti@mathi.uni-heidelberg.de}{gbenedetti@mathi.uni-heidelberg.de}} \par
  \addvspace{\medskipamount}
\noindent\textsc{Seoul National University, Department of Mathematical Sciences,  Research institute in Mathematics, Gwanak-Gu, 
	Seoul 08826, South Korea} \par  
  \textit{E-mail address}: \texttt{\href{mailto:jungsoo.kang@snu.ac.kr}{jungsoo.kang@snu.ac.kr}} \par
}}
\title{A local contact systolic inequality in dimension three}
\author{Gabriele Benedetti and Jungsoo Kang}
\begin{document}
\maketitle
\begin{abstract}
Let $\alpha$ be a contact form on a connected closed three-manifold $\Sigma$. The systolic ratio of $\alpha$ is defined as $\rho_{\mathrm{sys}}(\alpha):=\tfrac{1}{\Vol(\alpha)}T_{\min}(\alpha)^2$, where $T_{\min}(\alpha)$ and $\Vol(\alpha)$ denote the minimal period of periodic Reeb orbits and the contact volume. The form $\alpha$ is said to be Zoll if its Reeb flow generates a free $S^1$-action on $\Sigma$. We prove that the set of Zoll contact forms on $\Sigma$ locally maximises the systolic ratio in the $C^3$-topology. More precisely, we show that every Zoll form $\alpha_*$ admits a $C^3$-neighbourhood $\mathcal U$ in the space of contact forms such that, for every $\alpha\in\mathcal U$, there holds $\rho_{\mathrm{sys}}(\alpha)\leq \rho_{\mathrm{sys}}(\alpha_*)$ with equality if and only if $\alpha$ is Zoll.
\end{abstract}

\tableofcontents
\section{Introduction}
Let $\Sigma$ be a connected closed manifold of dimension $2n+1$ and let $\mathcal C(\Sigma)$ be the set of contact forms on it. Namely, the elements $\alpha\in\mathcal C(\Sigma)$ are one-forms on $\Sigma$ such that the $(2n+1)$-form $\alpha\wedge(\di\alpha)^{n}$ is nowhere vanishing. This property implies that there exists a unique vector field $R_\alpha$ on $\Sigma$ determined by the relations $\di\alpha(R_\alpha,\,\cdot\,)=0$ and $\alpha(R_\alpha)=1$.\label{def:reeb} The vector field $R_\alpha$ is called the Reeb vector field and the associated flow $\Phi^\alpha$ the Reeb flow. Periodic orbits of $\Phi^\alpha$ \label{def:phialpha} are fundamental objects in contact and symplectic geometry. The Weinstein conjecture, which asserts that every contact form on a closed manifold possesses at least one periodic orbit \cite{Wei79}, has played a prominent role in the field. The conjecture has been established in many particular situations, most notably when $\Sigma$ is three-dimensional \cite{Tau}. In these cases a more refined question arises: What can be said about the period of the orbits that one finds? A natural problem is, namely, to give an explicit upper bound on $T_{\min}(\alpha)$, the minimal period of periodic orbits of $\Phi^\alpha$, in terms of some geometric quantity associated with $\alpha$. Following \cite{APB14}, we use here the contact volume (other choices are possible and can lead to different results, as in \cite{AFM16})
\begin{equation}\label{e:cvol}
\Vol(\alpha):= \int_{\Sigma} \alpha\wedge (\di\alpha)^{n}>0,
\end{equation}
and consider the \textbf{systolic ratio}
\begin{equation*}
\rho_{\mathrm{sys}}:\mathcal C(\Sigma)\to(0,\infty],\qquad \rho_{\mathrm{sys}}(\alpha):= \frac{T_{\min}(\alpha)^{n+1}}{\Vol(\alpha)}.
\end{equation*}
Inside $\mathcal C(\Sigma)$ one can consider the subset $\mathcal C(\xi)$ of contact forms defining a given co-oriented contact structure $\xi$ on $\Sigma$. This means that $\xi$ is a co-oriented hyperplane field such that $\ker\alpha=\xi$ for all $\alpha\in\mathcal C(\xi)$. Following the breakthrough result in dimension three obtained by Abbondandolo, Bramham, Hryniewicz and Salom\~{a}o in \cite{ABHSnew}, Sa\u{g}lam showed that
\[
\sup_{\alpha\in \mathcal C(\xi)}\rho_{\mathrm{sys}}(\alpha)=+\infty,
\]
i.e.\ the systolic ratio does not admit a \textbf{global upper bound} on $\mathcal C(\xi)$, for any contact contact structure $\xi$ in any dimension \cite{Sag18}.

Such bound might hold, however, if one takes a special subclass of contact forms in $\mathcal C(\xi)$. For instance, a celebrated theorem of Viterbo \cite[Theorem 5.1]{Vit00} (see also \cite{AMO08}) asserts that the systolic ratio is bounded from above on the set of contact forms on $S^{2n+1}$ arising from convex embeddings into $\R^{2(n+1)}$.
Another distinguished subclass is given by the canonical contact forms on the unit tangent bundle of closed Riemannian or Finsler manifolds. This is the setting where systolic geometry originated and has been hitherto tremendously studied (see \cite[Chapter 7.2]{Ber03}).   

In a similar vein, for a general $\Sigma$, one is led to study the local behaviour of $\rho_{\mathrm{sys}}$ around its critical set. This direction of inquiry was initiated in \cite{APB14} by \'Alvarez-Paiva and Balacheff, who showed that $\op{Crit}\rho_{\mathrm{sys}}$ is exactly the set of Zoll contact forms.
\begin{dfn}\label{dfn:zoll}
A contact form $\alpha$ on a manifold $\Sigma$ is called \textbf{Zoll} of period $T(\alpha)>0$, if the flow $\Phi^{\alpha}$ induces a free $\R/T(\alpha)\Z$-action (all orbits are periodic and have prime period $T(\alpha)$). We write $\ZZ(\Sigma)$ for the set of all Zoll contact forms on $\Sigma$.
\end{dfn}
Once a Zoll form $\alpha_*$ is given, it is easy to deform it through a path $s\mapsto \alpha_s$ of Zoll forms with $\alpha_0=\alpha_*$. We can just set $\alpha_s:=T_s\Psi_s^*\alpha_*$, where $s\mapsto\Psi_s$ is any isotopy of $\Sigma$ and $s\mapsto T_s$ a path of positive numbers.  By a theorem of Weinstein \cite{Wei74}, these represent all possible deformations of $\alpha_*$ through Zoll contact forms.

While the local structure of $\ZZ(\Sigma)$ is well understood, describing when $\ZZ(\Sigma)$ is non-empty and investigating its global structure are more subtle issues. In this regard, a classical construction by Boothby and Wang represents a useful tool \cite[Theorem 2 and 3]{BW58}. From Definition \ref{dfn:zoll}, it follows that if $\alpha_*$ is Zoll of period $1$, the quotient by the action of the Reeb flow yields an oriented $S^1$-bundle $\mathfrak p:\Sigma\to M$, where $M$ is a closed manifold of dimension $2n$ and $S^1=\R/\Z$. The Zoll contact form $\alpha_*$ becomes a connection form for $\mathfrak p$, while the two-form $\di\alpha_*$ descends to a symplectic form $\om_*$ on $M$, representing minus the Euler class of $\mathfrak p$. Vice versa, given a symplectic manifold $(M,\om_*)$ such that the cohomology class of $\om_*$ is integral, one can construct an oriented $S^1$-bundle $\mathfrak p:\Sigma\to M$ with a connection form $\alpha_*$ satisfying $\di\alpha_*=\mathfrak p^*\om_*$, so that $\alpha_*$ is a Zoll form of period $1$ on $\Sigma$.

The Boothby-Wang construction tells us exactly which connected three-manifolds admit a Zoll contact form:~they are total spaces of non-trivial oriented $S^1$-bundles over connected oriented closed surfaces. In this case, an easy topological argument shows that the diffeomorphism type of the quotient $M$ and the Euler number of $\mathfrak p$ depend only on $\Sigma$ and not on $\alpha_*$. In particular, minus the Euler number equals $|H_1^\mathrm{tor}(\Sigma;\Z)|$, namely the cardinality of the torsion subgroup of the first integral homology of $\Sigma$. Hence, as shown in \cite[Proposition 3.3]{APB14}, we have the identity
\[
\rho_{\mathrm{sys}}(\alpha_*) = \frac{1}{|H_1^\mathrm{tor}(\Sigma;\Z)|}.
\]
The next result classifies Zoll contact forms on $\Sigma$ up to diffeomorphisms and up to isotopies. When $\Sigma$ is $SO(3)$ or $S^3$, the diffeomorphism classification was carried out in \cite[Theorem B.2]{ABHS17} and \cite[Proposition 3.9]{ABHS15}.

\begin{prp}\label{p:three}
Let $\Sigma$ be the total space of a non-trivial orientable $S^1$-bundle over a connected orientable closed surface.
\begin{enumerate}
	\item If $\alpha$ and $\alpha'$ are Zoll contact forms on $\Sigma$, there is a diffeomorphism $\Psi:\Sigma\to\Sigma$ and a positive constant $T>0$ such that
	\[
	\Psi^*\alpha'=T\alpha.
	\]
	\item The space $\ZZ(\Sigma)$ has exactly two connected components.
\end{enumerate}
\end{prp}
We provide a proof of the proposition together with a detailed description of the connected components of $\ZZ(\Sigma)$ in Section \ref{sec:class}.

Actually, the results in \cite{APB14} go beyond the characterisation of $\op{Crit}\rho_{\mathrm{sys}}$ and imply that if $s\mapsto\alpha_s$ is a smooth deformation of $\alpha_*\in\ZZ(\Sigma)$ with $\alpha_0=\alpha_*$, then $s\mapsto\rho_{\mathrm{sys}}(\alpha_s)$ attains a strict maximum at $0$, provided the deformation is not tangent in $s=0$ to all orders to $\ZZ(\Sigma)$. On the other hand, by Weinstein's theorem, if the deformation is contained in $\ZZ(\Sigma)$, then $\rho_{\mathrm{sys}}(\alpha_s)=\rho_{\mathrm{sys}}(\alpha_*)$ for all $s$. As communicated to us by the authors, an implicit goal in \cite{APB14} was to answer the following question on a sharp \textbf{local upper bound} for $\rho_{\mathrm{sys}}$.
\begin{con}[Local contact systolic inequality]\label{con:lsi}Let $\alpha_*$ be a Zoll contact form on a connected closed manifold $\Sigma$ of dimension $2n+1$ and let $k\geq 0$ be an integer. Does there exist a $C^k$-neighbourhood $\mathcal U$ of $\alpha_*$ in the set of contact forms on $\Sigma$ such that
\[
\rho_{\mathrm{sys}}(\alpha)\ \leq\ \rho_{\mathrm{sys}}(\alpha_*),\qquad \forall\, \alpha\in\mathcal U
\]
and the equality holds if and only if $\alpha$ is a Zoll form?
\end{con}
For Zoll Riemannian metrics on a compact rank one symmetric space, an analogous question was formulated in \cite{Bal,APB14}. This Riemannian question is answered positively for $S^2$ with  $k=2$ in \cite{ABHS17}.

In their seminal paper \cite{ABHS15}, Abbondandolo, Bramham, Hryniewicz, and Salom\~ao give a positive answer to Question \ref{con:lsi} with $k=3$ when $\Sigma$ is the three-sphere (or more generally, by means of a simple covering argument, when the base $M$ is the two-sphere). Moreover, they give a negative answer to the question in dimension three, if one replaces the $C^k$-closeness of contact forms with the $C^0_{\mathrm{loc}}$-closeness of the Reeb flows.

In the present paper, building on their beautiful result, we answer Question \ref{con:lsi} affirmatively with $k=3$ for every closed three-manifold admitting a Zoll contact form (so, compared with \cite{ABHS15}, here the base $M$ can be an arbitrary orientable closed surface), including a statement regarding the \textbf{diastolic ratio} 
\begin{equation*}\label{e:dia}
\rho_{\mathrm{dia}}(\alpha):=\frac{T_{\max}(\alpha)^{n+1}}{\Vol(\alpha)},
\end{equation*}
where $T_{\max}(\alpha)$ is the maximal period of \textit{prime} periodic orbits of $\Phi^\alpha$. If $\alpha$ is Zoll, then there holds $T_{\min}(\alpha)=T(\alpha)=T_{\max}(\alpha)$ so that $\rho_{\mathrm{sys}}(\alpha)=\rho_{\mathrm{dia}}(\alpha)$. To get a stronger result, for every free-homotopy class of loops $\mathfrak h$ in $\Sigma$, we also define the minimal and maximal period of prime periodic orbits of $\Phi^\alpha$ in the class $\mathfrak h$ and we denote them by $T_{\min}(\alpha,\mathfrak h)$ and $T_{\max}(\alpha,\mathfrak h)$, respectively. Finally, we write $\rho_{\mathrm{sys}}(\alpha,\mathfrak h)$ and $\rho_{\mathrm{dia}}(\alpha,\mathfrak h)$ for the corresponding systolic and diastolic ratios. Clearly, $\rho_{\mathrm{sys}}(\alpha)\leq\rho_{\mathrm{sys}}(\alpha,\mathfrak h)\leq \rho_{\mathrm{dia}}(\alpha,\mathfrak h)\leq\rho_{\mathrm{dia}}(\alpha)$.
\begin{thm}\label{t:main}
Let $\alpha_*$ be a Zoll contact form on a connected closed three-manifold $\Sigma$, and let  $\mathfrak h$ be the free-homotopy class of the prime periodic orbits of $\Phi^{\alpha_*}$. There exists a $C^2$-neighbourhood $\mathcal U$ of $\di\alpha_*$ in the space of exact two-forms on $\Sigma$ such that, for every contact form $\alpha$ on $\Sigma$ with $\di\alpha\in \mathcal U$, we have
\begin{equation*}
\rho_{\mathrm{sys}}(\alpha,\mathfrak h) \leq \frac{1}{|H_1^\mathrm{tor}(\Sigma;\Z)|} \leq \rho_{\mathrm{dia}}(\alpha,\mathfrak h)
\end{equation*}
and any of the two equalities holds if and only if $\alpha$ is Zoll. In particular, Zoll contact forms are strict local maximisers of the systolic ratio in the $C^3$-topology.
\end{thm}
\begin{rmk}
This result can be used to prove a systolic inequality for magnetic flows on closed oriented surfaces, as we discuss in \cite{BK19b}.
\end{rmk}
\noindent\textbf{Sketch of proof of Theorem \ref{t:main}.} The strategy of the proof closely follows the one in \cite{ABHS15}. We divide the proof of the theorem into two parts, corresponding to Section \ref{sec:global} and \ref{sec:generating}, respectively. 

In the \textbf{first part} we start by assuming without loss of generality that all prime orbits of $\alpha_*$ have period equal to $1$, the form $\alpha$ is $C^2$-close to $\alpha_*$ and $\di\alpha$ is $C^2$-close to $\di\alpha_*$. Then, we show that there exists a real number $T$ with $1<T<2$ such that the set $\mathcal P_T(\alpha,\mathfrak h)$ of prime periodic orbits $\gamma$ of $\Phi^\alpha$ in the class $\mathfrak h$ with period $T(\gamma)\leq T$ is not empty (see Proposition \ref{prp:gin}). Moreover, given $\gamma\in\mathcal P_T(\alpha,\mathfrak h)$, we construct a global surface of section $N\to\Sigma$ for $\Phi^\alpha$, which is diffeomorphic to $M$ with an open disc removed and such that its boundary covers $|H_1^\mathrm{tor}(\Sigma;\Z)|$-times the orbit $\gamma$ (see Section \ref{ss:surface}). If $\lambda$ is the restriction of $\alpha$ to $N$, then $\di\lambda$ is symplectic in the interior $\mathring N$ and vanishes of order one at the boundary $\p N$. The first-return time, a priori only defined on $\mathring N$, extends to a function $\tau:N\to(0,\infty)$, which is $C^1$-close to the constant $1$. The first-return map, a priori only defined on $\mathring N$, extends to a diffeomorphism $\varphi:N\to N$, which is $C^1$-close to $\id_N$. Moreover, there holds $\varphi^*\lambda-\lambda=\di \sigma$, where $\sigma:=\tau-T(\gamma)$ is a $C^1$-small function, called the action of $\varphi$. The volume of $\alpha$ is related to the Calabi invariant $\CAL(\varphi):=\tfrac12\int_N \sigma\di\lambda$ of the map $\varphi$ through the formula 
\begin{equation*}
\Vol(\alpha)=\int_N\tau\di\lambda=\int_N\big(\sigma+T(\gamma)\big)\di\lambda=2\CAL(\varphi)+|H_1^\mathrm{tor}(\Sigma;\Z)|\,T(\gamma)^2.
\end{equation*}
Furthermore, every fixed point $q\in \mathring N$ of $\varphi$ yields a periodic orbit $\gamma_q\in\mathcal P_T(\alpha,\mathfrak h)$ with period
\begin{equation*}
T(\gamma_q)=\sigma(q)+T(\gamma).
\end{equation*}
In particular, when $\alpha$ is not Zoll, $\varphi\neq \id_N$. The properties of the return time and the return map are collected in Theorem \ref{t:final3dim}. As a consequence, in Corollary \ref{c:neccond} we argue that Theorem \ref{t:main} is proven if we take $\gamma$ to have minimal, respectively, maximal period among orbits in $\mathcal P_T(\alpha,\mathfrak h)$, and are able to show that
\begin{equation}\label{e:implyintro}
\begin{aligned}
\varphi\neq\id_N,\ \ \CAL(\varphi)\leq 0\quad\Longrightarrow\quad \exists\, q_-\in\mathring N\cap\Fix(\varphi),\ \ \sigma(q_-)<0,\\[1ex]
\varphi\neq\id_N,\ \ \CAL(\varphi)\geq 0\quad\Longrightarrow\quad \exists\, q_+\in\mathring N\cap\Fix(\varphi),\ \ \sigma(q_+)>0.
\end{aligned}
\end{equation}
Indeed, if we take $\gamma\in\mathcal P_T(\alpha,\mathfrak h)$ with minimal period and assume $\rho_{\mathrm{sys}}(\alpha,\mathfrak h)\geq\frac{1}{|H_1^\mathrm{tor}(\Sigma;\Z)|}$, then $\CAL(\varphi)\leq 0$. But, if $\alpha$ is not Zoll, the first implication in \eqref{e:implyintro} yields $\gamma_{q_-}\in\mathcal P_T(\alpha,\mathfrak h)$ with $T(\gamma_{q_-})<T(\gamma)$. This contradiction proves $\rho_{\mathrm{sys}}(\alpha,\mathfrak h)<\frac{1}{|H_1^\mathrm{tor}(\Sigma;\Z)|}$ for a contact form $\alpha$ which is not Zoll. The inequality for $\rho_{\mathrm{dia}}(\alpha,\mathfrak h)$ follows analogously.

In the \textbf{second part} of the proof, we establish implications \eqref{e:implyintro}. When $N$ is the two-disc, these implications were already shown to hold in \cite[Corollary 5]{ABHS15}. The key step there is to find a formula for the Calabi invariant in terms of the generating function of $\varphi$ \cite[Proposition 2.20]{ABHS15}. Instead of finding such a formula for general $N$, we show implications \eqref{e:implyintro} by constructing a path $t\mapsto\varphi_t$ of $\di\lambda$-Hamiltonian diffeomorphisms of $N$ with $\varphi_0=\id_N$, $\varphi_1=\varphi$, which is generated by a quasi-autonomous Hamiltonian $H:N\times[0,1]\to\R$ (see Proposition \ref{p:qa}). We recall from \cite{BP94} that a function $H$ is quasi-autonomous if there exist $q_{\min},q_{\max}\in N$ such that
\[
\min_{q\in N} H(q,t)=H(q_{\min},t),\qquad \max_{q\in N} H(q,t)=H(q_{\max},t),\qquad \forall\, t\in[0,1].
\]
In particular, $q_{\min}$ and $q_{\max}$ are fixed points of $\varphi$, if they lie in $\mathring N$. In order to exhibit such a path, we construct a Weinstein neighbourhood of the diagonal in $\big(N\times N,(-\di\lambda)\oplus\di\lambda\big)$ (see Proposition \ref{prop:neighborhood_map}). This yields a generating function $G:N\to\R$ for $\varphi$. Let $[0,\epsilon)\times S^1\subset N$ be a collar neighbourhood of the boundary with radial coordinate $R$. At this point, crucially using that $\di\lambda$ vanishes at $\partial N$ of order one in the radial direction, we can show (see Proposition \ref{p:contg}) that the generating function belongs to
\[
\mathbb G:=\Big\{G:N\to \R\ \Big|\ G=0 \text{ on }\partial N,\ \ G \text{ is $C^2$-small on }N,\ \ \tfrac1R\di G\text{ is $C^1$-small on }[0,\epsilon)\times S^1\Big\}.
\]
Conversely, every $G\in\mathbb G$ is the generating function of some diffeomorphism $\varphi_{G}:N\to N$, which is $C^1$-close to the identity (see Proposition \ref{prp:bijectivity_of_Xi}). Therefore, since the set $\mathbb G$ is star-shaped around the zero function, the Hamilton-Jacobi equation (see \eqref{e:hj}) tells us that, for every $\varphi=\varphi_G$, the Hamiltonian function $H:N\x[0,1]\to\R$ associated with the path $t\mapsto \varphi_{tG}$, $t\in[0,1]$, is quasi-autonomous. Once the existence of a quasi-autonomous Hamiltonian is settled, implications \eqref{e:implyintro} follow (see Corollary \ref{cor:implications}), as already observed in \cite[Remark 2.8]{ABHS15}. Indeed, we can rewrite the Calabi invariant of $\varphi$ and the action of $q_{\min}$ (and similarly of $q_{\max}$), provided it lies in $\mathring N$, as
\begin{equation}\label{e:calxmin}
\CAL(\varphi)=\int_{N\times[0,1]}H\,\di\lambda\wedge\di t,\qquad \sigma(q_{\min})=\int_0^1 H(q_{\min},t)\di t.
\end{equation}
This finishes the second part of the proof and the whole sketch.
\begin{rmk}
Relations \eqref{e:calxmin} suggest that one could interpret implications \eqref{e:implyintro} as a local systolic (resp.\ diastolic) inequality for quasi-autonomous Hamiltonian diffeomorphisms. Such an inequality yields an upper (resp.\ lower) bound on the minimal (resp.\ maximal) action of a contractible fixed point in terms of the Calabi invariant. The bound can be readily proven for closed symplectic manifolds in arbitrary dimension. On the other hand, Reeb flows and Hamiltonian diffeomorphisms are two special incarnations of the characteristic foliation of an odd-symplectic form (also known as a Hamiltonian structure \cite{CM05}) on an oriented circle bundle over a closed symplectic manifold. These observations prompted us to formulate a conjectural systolic inequality for odd-symplectic forms, which we discuss in \cite{BK19}.
\end{rmk}

\noindent\textbf{Acknowledgements.} This work is part of a project in the Collaborative Research Center \textit{TRR 191 - Symplectic Structures in Geometry, Algebra and Dynamics} funded by the DFG. It was initiated when the authors worked together at the University of M\"unster and partially carried out while J.K.~was affiliated with the Ruhr-University Bochum. We thank Peter Albers, Kai Zehmisch, and the University of M\"unster for having provided an inspiring academic environment. We are grateful to Alberto Abbondandolo for valuable discussions and suggestions. We are indebted to the anonymous referee for the careful reading of the manuscript and for helpful comments on its first draft. G.B.~would like to express his gratitude to Hans-Bert Rademacher and the whole Differential Geometry group at the University of Leipzig.  G.B.~was supported by the National Science Foundation under Grant No. DMS-1440140 while in residence at the Mathematical Sciences Research Institute in Berkeley, California,
during the Fall 2018 semester. J.K.~is supported by Samsung Science and Technology Foundation (SSTF-BA1801-01).

\section{Classification of Zoll contact forms in dimension three}\label{sec:class}

This section is devoted to establish Proposition \ref{p:three}. For a clear exposition, we divide the proof into two lemmas. In the first one, we show that all Zoll contact forms on $\Sigma$ are isomorphic. This was proved in \cite[Theorem B.2]{ABHS17} when $\Sigma=SO(3)$ and in \cite[Proposition 3.9]{ABHS15} when $\Sigma=S^3$.
\begin{lem}\label{lem:Zoll_uniqueness}
Let $\Sigma$ be a connected closed three-manifold. Let $\alpha$ and $\alpha'$ be two Zoll contact forms on $\Sigma$ with unit period. There exists a diffeomorphism $\Psi:\Sigma\to \Sigma$ such that
\begin{equation*}
\Psi^*\alpha' = \alpha.
\end{equation*} 
\end{lem}
\begin{proof}
The Reeb flows of $\alpha$ and $\alpha'$ yield $S^1$-actions on $\Sigma$ and let $\mathfrak p:\Sigma\to M$ and $\mathfrak p':\Sigma\to M'$ be the associated oriented $S^1$-bundles. We write $e$ and $e'$ for minus the real Euler class of $\mathfrak p$ and $\mathfrak p'$. Let us orient $M$ and $M'$ through the forms $\om$ and $\om'$, where $\di\alpha=\mathfrak p^*\om$ and $\di\alpha'=\mathfrak p'^*\om'$. By a standard topological argument, the surfaces $M$ and $M'$ have the same genus and $\langle e,[M]\rangle=\langle e',[M']\rangle$. As the Euler number $\langle e,[M]\rangle$ is a complete invariant for principal $S^1$-bundles over oriented surfaces, there exists an $S^1$-equivariant diffeomorphism $\Psi_1:\Sigma\to\Sigma$ such that $\mathfrak p'\circ\Psi_1=\psi_1\circ\mathfrak p$, for some orientation-preserving diffeomorphism $\psi_1:M\to M'$. As a result, if $\alpha_1:=\Psi^*_1\alpha'$, then there exists a one-form $\eta$ on $M$ such that $\alpha_1=\alpha+\mathfrak p^*\eta$ and $\di\alpha_1=\mathfrak p^*\omega_1$, where $\omega_1:=\psi_1^*\om'$. We construct now a diffeomorphism $\Psi_2:\Sigma\to\Sigma$ with the property $\Psi_2^*\alpha_1=\alpha$, so that $\Psi:=\Psi_2\circ\Psi_1$ is the desired map. Using a stability argument, we seek an isotopy $\Phi_u:\Sigma\to\Sigma$ generated by a vector field $X_u$ such that
\begin{equation}\label{e:gray}
\Phi_u^*\alpha_u=\alpha, 
\end{equation}
where $\alpha_u:=\alpha+u\,\mathfrak p^*\eta$, for all $u\in[0,1]$. We will then set $\Psi_2:=\Phi_1$. We observe that $\omega_u:=(1-u)\om+u\om_1$ is a path of symplectic forms on $M$, as $\psi_1$ preserves the orientation. Differentiating \eqref{e:gray} with respect to $u$, we see that \eqref{e:gray} is satisfied once $X_u$ is chosen as the vector field in $\ker\alpha_u$ with the property that $\di\mathfrak p(X_u)=\bar X_u$, where $\bar X_u$ is the unique vector field on $M$ satisfying the relation $\iota_{\bar X_u}\omega_u=-\eta$.
\end{proof}
Recall that $\mathcal Z(\Sigma)$ is the space of Zoll contact forms on $\Sigma$. Let $\bm{\xi}$ be an isotopy class of co-oriented contact structures on $\Sigma$, and let $\mathcal Z(\bm\xi)$ be the set of all Zoll forms defining some element in $\bm\xi$. We denote by $-\bm\xi$ the isotopy class obtained by reversing the co-orientation of the contact structures in $\bm\xi$. 
\begin{lem}
Let $\Sigma$ denote the total space of a non-trivial orientable $S^1$-bundle over a connected closed orientable surface.
\begin{enumerate}
\item If $\Sigma$ is either $S^3$ or $\R\mathbb P^3$, then $\ZZ(\Sigma)$ has exactly two connected components $\ZZ(\bm\xi_{\op{st}})$ and $\ZZ(\bm\xi_{\overline{\op{st}}})$. Here $\bm\xi_{\op{st}}$ is the isotopy class of the standard contact structure and $\bm\xi_{\overline{\op{st}}}$ the isotopy class obtained from $\bm\xi_{\op{st}}$ by applying an orientation-reversing diffeomorphism.
\item If $\Sigma$ is neither $S^3$ nor $\R\mathbb P^3$, then $\ZZ(\Sigma)$ has exactly two connected components $\ZZ(\bm\xi_+)$ and $\ZZ(\bm\xi_-)$. Here $\bm\xi_+$ and $\bm\xi_-$ are two distinct isotopy classes with $\bm\xi_-=-\bm\xi_+$.
\end{enumerate} 
\end{lem}
\begin{proof}
Let us fix a Zoll contact form $\alpha$ on $\Sigma$ with unit period and bundle map $\mathfrak p:\Sigma\to M$. We consider any other Zoll form $\alpha'$ with unit period on $\Sigma$ and we distinguish two cases.\\[-2ex]

\noindent\textbf{Case 1: $M=S^2$.} Here $\Sigma$ is the lens space $L(p,1)$ for some $p\geq 1$. Lemma \ref{lem:Zoll_uniqueness} yields a diffeomorphism $\Psi:\Sigma\to\Sigma$ with the property $\alpha'=\Psi^*\alpha$. Suppose that $\Sigma$ is either $S^3$ or $\R\mathbb P^3$ and let $\bar\Upsilon:\Sigma\to\Sigma$ be a diffeomorphism of $\Sigma$ reversing the orientation. By Cerf's Theorem (see \cite{Cer68}, and \cite[Th\'eor\`eme 3]{Bon83} or \cite[Theorem 5.6]{HR85}), $\Psi$ is either isotopic to the identity or to $\bar\Upsilon$, thus showing that $\alpha'$ is either homotopic to $\alpha$ or to $\bar{\Upsilon}^*\alpha$ within $\ZZ(\Sigma)$. 

Suppose now that $\Sigma$ is neither $S^3$ nor $\R\mathbb P^3$. Then, a $\mathfrak p$-fibre is not homotopic to itself with reverse orientation. Therefore, $\alpha$ and $-\alpha$ are not homotopic in $\ZZ(\Sigma)$. Moreover, by Lemma \ref{lem:Zoll_uniqueness}, there exists a diffeomorphism $\Upsilon_-:\Sigma\to\Sigma$ such that $\Upsilon_-^*\alpha=-\alpha$. In particular, $\Upsilon_-$ changes the orientation of the fibres and is not isotopic to the identity. By \cite[Th\'eor\`eme 3]{Bon83} or \cite[Theorem 5.6]{HR85} again, the map $\Psi$ is either isotopic to the identity or to $\Upsilon_-$. Hence, $\alpha'$ is either homotopic to $\alpha$ or to $-\alpha$ within $\ZZ(\Sigma)$.\\[-2ex]

\noindent\textbf{Case 2: $M\neq S^2$.} The long exact sequence of homotopy groups shows that a $\mathfrak p$-fibre is not homotopic to itself. Therefore, $\alpha$ and $-\alpha$ are not homotopic within $\ZZ(\Sigma)$. Moreover, \cite[Satz 5.5]{Wal67} implies that there exists a diffeomorphism $\Psi:\Sigma\to\Sigma$ isotopic to the identity and such that $\Psi^*\alpha'$ is an $S^1$-connection for $\mathfrak p$ or for $\mathfrak p$ with reversed orientation. The stability argument contained in the proof of Lemma \ref{lem:Zoll_uniqueness} shows that $\Psi^*\alpha'$ is homotopic to $\alpha$ or $-\alpha$.
\medskip

We finally observe that if $\Sigma$ is not $S^3$ nor $\R\mathbb P^3$, then $\bm\xi_+$ and $\bm\xi_-$ are not isotopic. To this purpose, we use the last statement of Theorem D in \cite{Mas08}. The fact that $\Sigma\neq S^3,\R\mathbb P^3$ is equivalent to the fact that $\langle e,[M]\rangle> \chi(M)$, and implies the hypothesis $-b-r<2g-2$ contained therein, where $b=\langle e,[M]\rangle$, $r=0$ and $2g-2=-\chi(M)$. Therefore, one only needs to check that the twisting number $t(\bm\xi_\pm)$ defined in \cite[p.~1730]{Mas08} is equal to $-1$. If we suppose that $\alpha$ has period $1$, then it is an $S^1$-connection for $\mathfrak p$ with $\di\alpha=\mathfrak p^*\om$ and there exists a positively immersed disc $D^2\hookrightarrow M$, whose lift to the universal cover of $M$ is embedded and such that $\int_{D^2}\om=1$. One readily sees that the horizontal lift of the boundary of $D^2$ traversed in the negative direction is a $\bm\xi_\pm$-Legendrian curve in $\Sigma$, which is isotopic to an oriented $\mathfrak p$-fibre and has twisting number $-1$. 
\end{proof}
\section{A global surface of section for contact forms near Zoll ones}\label{sec:global}
Let us start by fixing some notation which will be used below. As before, we set $S^1= \R/\Z$. Let $\Sigma$ be a connected closed three-manifold and let $\alpha_*$ be a Zoll contact form on $\Sigma$ with unit period (see Definition \ref{dfn:zoll}). Let $R_*$ denote the Reeb vector field of $\alpha_*$. Since $\alpha_*$ is Zoll, $R_*$ induces a free $S^1$-action on $\Sigma$ and yields an oriented $S^1$-bundle $\mathfrak p:\Sigma\to M$, where $M$ is the quotient of $\Sigma$ by the action and $\mathfrak p$ is the canonical projection. We write \label{def:mathfrakh1}$\mathfrak h$ for the free-homotopy class of the oriented $\mathfrak p$-fibres. Throughout this section, we fix auxiliary Riemannian metrics on $\Sigma$ and $M$, in order to compute the distance between points and between diffeomorphisms, and the norm of sections of vector bundles over these manifolds. The space $M$ is a connected closed surface having a symplectic form $\om_*$ satisfying 
\begin{equation*}\label{eq:alphaomega}
\di\alpha_*=\mathfrak p^*\om_*.
\end{equation*}
We endow $M$ with the orientation induced by $\om_*$.

Let $g_\st$ and $i$ be the standard scalar product and complex structure on $\R^2\cong \C$, respectively. If $a>0$ is an arbitrary positive number, we denote by $B$ (respectively $B'$) the closed Euclidean ball in $\R^{2}$ of radius $a$ (respectively $a/2$). We write $x=(x_1,x_2)$ for a point in $B$ and let $\lambda_\st=\tfrac{1}{4\pi}(x_1\di x_2-x_2 \di x_1)$ be the standard Liouville form (up to a constant) on $B$. We consider the trivial bundle $\mathfrak p_\st:B\times S^1\to B$ and we write $\phi$ for the fibre coordinate. We set
\begin{equation*}\label{e:standardstructures}
\alpha_\st:=\di\phi+\mathfrak p_\st^*\lambda_\st,\qquad R_\st:=\partial_\phi.
\end{equation*}
We now define a finite Darboux covering for $M$. To this purpose, let $Z\subset \Sigma$ be a finite set of points. We consider $S^1$-equivariant embeddings
\begin{equation*}\label{eq:H_z}
\Df_z:B\x S^1\longrightarrow \Sigma,\qquad \Df_z(0,0)=z,\qquad\forall\, z\in Z.
\end{equation*}
This means that there are corresponding embeddings
\[
\df_q:B\longrightarrow M,\qquad \df_q(0)=q,\qquad\forall\, q\in \mathfrak p(Z)
\] 
such that
\[
\mathfrak p\circ \Df_z=\df_{\mathfrak p(z)}\circ\mathfrak p_\st,\qquad \forall\, z\in Z.
\]
We write \label{def:sigmaz}$\Sigma_z:=\Df_z(B\times S^1)$, $\Sigma'_z:=\Df_z(B'\times S^1)$, and $M_q:=\df_q(B)$, $M'_q:=\df_q(B')$. Finally, we denote by $(x_z,\phi_z)\in B\x S^1$ the coordinates given by $\Df_z$. By the compactness of $\Sigma$, we see that, if $a$ is small enough, the following three properties can be assumed to hold
\begin{equation}\label{e:df1}
\begin{aligned}
{\bf (DF1)}&\quad M=\bigcup_{q\in \mathfrak p(Z)}M'_q,\\
{\bf (DF2)}&\quad \exists\, d_*>0,\quad \dist(M'_q,M\setminus M_q)>d_*,\quad\forall\, q\in \mathfrak p(Z),\\[1ex]
{\bf (DF3)}&\quad \Df_z^*\alpha_*=\alpha_{\op{st}},\quad \forall\, z\in Z.
\end{aligned}
\end{equation}
In this section, we define a neighbourhood of $\di\alpha_*$ in the space of exact two-forms on $\Sigma$ with special properties. The elements of the neighbourhood will be exterior differentials of contact forms, whose Reeb flow has a distinguished set of periodic Reeb orbits, which can be used to construct a global surface of section for the flow.
\subsection{A distinguished class of periodic Reeb orbits}\label{ss:neighbourhood}

For any contact form $\alpha$ on $\Sigma$, let $R_\alpha$ be its Reeb vector field. Let $\mathcal P(\alpha)$ denote the set of prime periodic orbits of the Reeb flow $\Phi^\alpha$ of $\alpha$. For all $T\in(0,\infty)$, we also denote by $\mathcal P_T(\alpha)$ the subset of $\mathcal P(\alpha)$, whose elements have period less than or equal to $T$. We write $\mathcal P(\alpha,\mathfrak h)$ for the subset of $\mathcal P(\alpha)$, whose elements are in the class $\mathfrak h$, namely they are freely homotopic to an oriented $\mathfrak p$-fibre. We abbreviate $\mathcal P_T(\alpha,\mathfrak h):=\mathcal P(\alpha,\mathfrak h)\cap \mathcal P_T(\alpha)$.

If $\gamma\in\mathcal P(\alpha)$, we write $T(\gamma)$ for the period of $\gamma$ and define the auxiliary one-periodic curves
\begin{equation*}\label{def:gammarepbar}
\gamma_\rep,\bar\gamma:S^1\to\Sigma,\qquad \gamma_\rep(u):=\gamma(uT(\gamma)),\qquad \bar\gamma(u):=\Phi^{\alpha_*}_u(\gamma(0)).
\end{equation*}
We define
\begin{equation*}\label{e:tminmax}
T_{\min}(\alpha,\mathfrak h):=\inf_{\gamma\in\mathcal P(\alpha,\mathfrak h)}T(\gamma),\qquad T_{\max}(\alpha,\mathfrak h):=\sup_{\gamma\in\mathcal P(\alpha,\mathfrak h)}T(\gamma).
\end{equation*}
We now explore how much information of the Reeb dynamics is already encoded in the exterior differential of the contact form.
\begin{lem}\label{l:neighbourhood1}
Let $\alpha_1$ and $\alpha_2$ be contact forms such that $\di\alpha_1=\di\alpha_2$. The forms $\alpha_1\wedge\di\alpha_1$ and $\alpha_2\wedge\di\alpha_2$ induce the same orientation on $\Sigma$ and $\Vol(\alpha_1)=\Vol(\alpha_2)$. Moreover, there is a bijection between $\mathcal P(\alpha_1)$ and $\mathcal P(\alpha_2)$ which preserves the oriented support of curves. The bijection is period-preserving when restricted to $\mathcal P(\alpha_1,\mathfrak h)$ and $\mathcal P(\alpha_2,\mathfrak h)$. If $\alpha_1$ is Zoll, then $\alpha_2$ is also Zoll, and $T(\alpha_1)=T(\alpha_2)$.
\end{lem}
\begin{proof}
Since $\di\alpha_1=\di\alpha_2$, we have $R_{\alpha_2}=\tfrac{1}{\alpha_2(R_{\alpha_1})}R_{\alpha_1}$ and $\alpha_2=\alpha_1+\eta$ for some closed one-form $\eta$. We orient $\Sigma$ so that $\alpha_2\wedge\di\alpha_2$ is positive and compute
\[
\Vol(\alpha_2)=\int_\Sigma \alpha_2\wedge \di\alpha_2=\int_\Sigma \alpha_1\wedge \di\alpha_1+\int_\Sigma\eta\wedge \di\alpha_2=\int_\Sigma \alpha_1\wedge \di\alpha_1=\Vol(\alpha_1).
\]
In particular, the orientations induced by $\alpha_1$ and by $\alpha_2$ coincide. Therefore, for all $z\in\Sigma$, there holds
\begin{equation*}
\Phi^{\alpha_2}_{t_2(t_1,z)}(z)=\Phi^{\alpha_1}_{t_1}(z),\qquad t_2(t_1,z):=\int_0^{t_1}\big(t\mapsto\Phi^{\alpha_1}_t(z)\big)^*\alpha_2,
\end{equation*}
so that $t_1\mapsto t_2(t_1,z)$ is strictly increasing. Hence, $\Phi^{\alpha_1}$ and $\Phi^{\alpha_2}$ have the same trajectories, up to an orientation-preserving reparametrisation, and we have a bijective correspondence between $\mathcal P(\alpha_1)$ and $\mathcal P(\alpha_2)$ preserving the oriented support of periodic orbits. Let $\gamma_1\in\mathcal P(\alpha_1,\mathfrak h)$ and $\gamma_2\in\mathcal P(\alpha_2,\mathfrak h)$ be corresponding periodic orbits. Since the homology class of $\gamma_1$ and $\gamma_2$ is torsion, the fact that $\eta$ is closed implies
\[
T(\gamma_2)=\int_{\R/T(\gamma_2)\Z}\gamma_2^*\alpha_2=\int_{\R/T(\gamma_1)\Z}\gamma_1^*\alpha_1+\int_{\R/T(\gamma_2)\Z}\gamma_2^*\eta=T(\gamma_1)+0.
\]
Finally, if $\alpha_1$ is Zoll with period $T_1$, then $\alpha_2$ is also Zoll with period $T_2:=t_2(T_1,z)$ (independent of $z\in\Sigma$), as $t_1\mapsto t_2(t_1,z)$ is monotone increasing. Since every prime periodic orbit of $\Phi^{\alpha_1}$ has torsion homology class, we conclude as above that $T_2=T_1$.
\end{proof}
On the space of one-forms $\alpha$ on $\Sigma$ we consider the norm $\Vert \cdot\Vert_{C^3_-}$ defined by
\[\label{e:normalpha}
\Vert \alpha\Vert_{C^3_-}:=\Vert \alpha\Vert_{C^2}+\Vert \di\alpha\Vert_{C^2}.
\]
There is a constant $C_{\mathfrak D}>0$ depending only on $\Sigma$ and the Darboux family such that for every one-form $\alpha$ on $\Sigma$,
\begin{equation}\label{e:constantcd}
\frac{1}{C_{\mathfrak D}}\Vert\alpha\Vert_{C^3_-}\leq \max_{z\in Z}\Vert \mathfrak D_z^*\alpha\Vert_{C^3_-}\leq C_{\mathfrak D}\Vert\alpha\Vert_{C^3_-}. 
\end{equation}
For every $\epsilon>0$, we denote the $C^3_-$-ball with center $\alpha_*$ and radius $\epsilon$ by
\[\label{e:ball}
\mathcal B(\epsilon):=\big\{\alpha\ \text{one-form on }\Sigma\ \big|\ \Vert\alpha-\alpha_*\Vert_{C^3_-}<\epsilon\big\}.
\]
The next result shows why it is natural to consider the $C^3_-$-norm for our purposes.
\begin{lem}\label{l:estomega}
There exists a constant $C_0>0$ such that for all one-forms $\alpha'$ on $\Sigma$, there is a one-form $\alpha$ on $\Sigma$ with the property that 
\begin{equation*}
\bullet\quad \di\alpha=\di\alpha',\qquad\qquad \bullet\quad \forall\,\epsilon>0,\quad \Vert \di\alpha'-\di\alpha_*\Vert_{C^2}<\epsilon\quad\Longrightarrow\quad \alpha\in\mathcal B(C_0\epsilon).
\end{equation*}
\end{lem}
\begin{proof}
By standard elliptic arguments (see for instance,  \cite[Chapter 10]{Nic07}), there exists a constant $C'_0>0$ such that for any exact two-form $\Omega$ on $\Sigma$, we can find a one-form $\eta_\Omega$ with
\[
\di\eta_\Omega=\Omega,\qquad \Vert \eta_\Omega\Vert_{C^2}\leq C'_0\Vert \Omega\Vert_{C^2}.
\]
Setting $\alpha:=\alpha_*+\eta_{\di\alpha'-\di\alpha_*}$ and applying the above fact to $\eta_{\di\alpha'-\di\alpha_*}$, we have $\di\alpha=\di\alpha'$ and 
\[
\Vert \alpha-\alpha_*\Vert_{C^2}\leq C'_0\Vert\di\alpha'-\di\alpha_*\Vert_{C^2},\qquad \Vert \di\alpha-\di\alpha_*\Vert_{C^2}=\Vert\di\alpha'-\di\alpha_*\Vert_{C^2}.
\]
The statement follows with $C_0:=C_0'+1$.
\end{proof}
We can now proceed to study the Reeb dynamics for one-forms in the sets $\mathcal B(\epsilon)$.
\begin{lem}\label{lem:gin}
There exist $\epsilon_0>0$ and $C_1>0$ with the following properties. Every $\alpha\in\mathcal B(\epsilon_0)$ is a contact form, and if $z'\in\Sigma$, $z\in Z$ and $T\in(0,\infty)$ are such that the integral curve $t\mapsto \Phi^{\alpha}_t(z')$ lies in $\Sigma_z$ for all $t\in[0,T]$, then the curve $\gamma_z:=(x_z(t),\phi_z(t))=\Df_z^{-1}(\Phi^{\alpha}_t(z))$ satisfies
\begin{equation}\label{eq:estimates1}
\|\dot \gamma_z-R_\st \|_{C^{2}}\leq C_1\|\alpha-\alpha_*\|_{C^3_-}.
\end{equation}
Thus, if $\wh \phi_{z}:[0,T]\to\R$ with $\wh \phi_{z}(0)=0$ is a lift of $\phi_z-\phi_z(0)$, then
\begin{equation}\label{eq:estimates2}
 |x_z(t)-x_z(0)|\leq C_1t\|\alpha-\alpha_*\|_{C^3_-},\qquad |\wh\phi_z(t)-t|\leq C_1t\|\alpha-\alpha_*\|_{C^3_-},\quad \forall\, t\in [0,T].
\end{equation}
\end{lem}
\begin{proof}
If $\alpha$ is a one-form on $\Sigma$ and we set $\alpha_z:=\Df_z^*\alpha$, then the estimate \eqref{e:constantcd} yields
\begin{equation}\label{e:estimatez}
\Vert \alpha_z-\alpha_{\op{st}}\Vert_{C^3_-}\leq C_{\mathfrak D} \|\alpha-\alpha_*\|_{C^3_-}.
\end{equation}
If $\epsilon_0>0$ is sufficiently small, then every $\alpha\in\mathcal B(\epsilon_0)$ is a contact form and there exists $A>0$ such that
\begin{equation}\label{e:ralphaz}
\Vert R_{\alpha_z}-R_{{\op{st}}}\Vert_{C^2}\leq A\Vert\alpha_z-\alpha_{\op{st}}\Vert_{C^3_-},\qquad \forall\,\alpha\in\mathcal B(\epsilon_0).
\end{equation}
Moreover, using \eqref{e:estimatez}, we have
\begin{equation}\label{e:inrz}
\Vert R_{\alpha_z}-R_\st\Vert_{C^2}\leq AC_{\mathfrak D}\Vert \alpha-\alpha_*\Vert_{C^3_-}.
\end{equation}
Therefore, we just need to estimate the left-hand side of \eqref{eq:estimates1} against $\Vert R_{\alpha_z}-R_\st\Vert_{C^2}$. We know that $\dot\gamma_z=R_{\alpha_z}(\gamma_z)$, which yields $\|\dot \gamma_z-R_\st \|_{C^0}\leq \|R_{\alpha_z}-R_\st \|_{C^2}$. For the higher derivatives, we just observe that $\Vert\dot\gamma_z\Vert_{C^0}$ is uniformly bounded by $1+AC_{\mathfrak D}$ and
\begin{align*}
\frac{\di}{\di t}(\dot\gamma_z-R_\st)&=\ddot\gamma_z=\di_{\gamma_z} R_{\alpha_z}\cdot\dot\gamma_z=\di_{\gamma_z} (R_{\alpha_z}-R_\st)\cdot\dot\gamma_z;\\
\frac{\di^2}{\di t^2}(\dot\gamma_z-R_\st)&=\di^2_{\gamma_z} R_{\alpha_z}(\dot\gamma_z,\dot\gamma_z)+\di_{\gamma_z}R_{\alpha_z}\cdot\ddot\gamma_z\\
&=\di^2_{\gamma_z} (R_{\alpha_z}-R_\st)(\dot\gamma_z,\dot\gamma_z)+\di_{\gamma_z}(R_{\alpha_z}-R_\st)\frac{\di}{\di t}(\dot\gamma_z-R_\st).
\end{align*}
This shows \eqref{eq:estimates1}. Finally, integrating $\dot \phi_z$ and $\dot x_z$ and using \eqref{eq:estimates1}, we obtain \eqref{eq:estimates2}.
\end{proof}
\begin{prp}\label{prp:gin}
There exist $C_2>0$, and for all real numbers $T$ in the interval $(1,2)$, a radius $\epsilon_1=\epsilon_1(T)\in(0,\epsilon_0]$ such that for all $\alpha\in\mathcal B(\epsilon_1)$ the following properties are true:
\begin{enumerate}[(i)]
\item A periodic orbit $\gamma$ of $\Phi^\alpha$ belongs to the set $\mathcal P_T(\alpha)$ if and only if for all $z\in Z$ such that $\gamma(0)\in\Sigma'_z$, then $\gamma$ is contained in $\Sigma_{z}$ and $\gamma_\rep$ is homotopic to $\bar\gamma$ within $\Sigma_{z}$. In this case, if we set $\gamma_z:=\Df_z^{-1}\circ \gamma$, $\bar\gamma_z:=\Df_z^{-1}\circ\bar\gamma$, there holds
\begin{equation*}
\big|T(\gamma)-1\big|\leq C_2\|\alpha-\alpha_*\|_{C^3_-},\qquad \|{\gamma}_{z,\rep}-\bar\gamma_z\|_{C^3}\leq C_2\|\alpha-\alpha_*\|_{C^3_-}.
\end{equation*}
\item The set $\mathcal P_T(\alpha,\mathfrak h)$ is compact, non-empty and coincides with $\mathcal P_T(\alpha)$.
\end{enumerate}  
\end{prp}
\begin{proof}
We claim that item (i) holds with
\[
\epsilon_1:=\frac{1}{2C_1}\min\Big\{d_*,2-T,T-1\Big\},\qquad C_2:=10\cdot C_1.
\]
Moreover, if a periodic curve $\gamma$ is contained in $\Sigma_z$, then we can write $\gamma=(x_\gamma,\phi_\gamma)$ in the coordinates $\Df_{z}$. Moreover, if $\wh \phi_{\gamma}:\R\to\R$ is the unique lift of $\phi_\gamma-\phi_\gamma(0)$ such that $\wh \phi_{\gamma}(0)=0$, then $\wh\phi_\gamma(T(\gamma))=1$ if and only if $\gamma_\rep$ is homotopic to $\bar\gamma$ within $\Sigma_z$.

Let us now assume that $\alpha\in\mathcal B(\epsilon_1)$ and that $\gamma\in\mathcal P_{T}(\alpha)$. Let us take $z\in Z$ such that $\gamma(0)\in \Sigma'_z$. Since $T<2$, inequalities \eqref{eq:estimates2} and $\bf (DF2)$ imply that $\gamma$ is contained in $\Sigma_z$. By \eqref{eq:estimates1}, we see that
\[
\dot\phi_\gamma\geq 1-|1-\dot\phi_\gamma|\geq 1-\Vert \dot\gamma_z-R_{\mathrm{st}}\Vert_{C^2}> 1-C_1\epsilon_1\geq \tfrac12>0.
\]
Hence, $\wh\phi_\gamma\big(T(\gamma)\big)>0$. On the other hand, using \eqref{eq:estimates2} and the fact that $\epsilon_1\leq \tfrac{2-T}{2C_1}$, we get
\[
\wh\phi_\gamma\big(T(\gamma)\big)< T(\gamma)+C_1T\epsilon_1\leq  T+2C_1\epsilon_1\leq 2.
\]
Since $\wh\phi_\gamma\big(T(\gamma)\big)$ is an integer, we conclude that $\wh\phi_\gamma\big(T(\gamma)\big)=1$.

Conversely, we assume that $\gamma=(x_\gamma,\phi_\gamma)\subset \Sigma_z$ and that $\gamma_{\mathrm{rep}}$ is homotopic to $\bar\gamma$ inside $\Sigma_z$ and prove that $\gamma\in\mathcal P_T(\alpha)$. The curve $\gamma$ is prime since $\wh\phi_\gamma\big(T(\gamma)\big)=1$ has no non-trivial integer divisor. Substituting $t=T(\gamma)$ in the second inequality in \eqref{eq:estimates2} yields 
\begin{equation}\label{eq:period_1}
|T(\gamma)-1|\leq C_1 T(\gamma)\|\alpha-\alpha_*\|_{C^3_-}.
\end{equation}
Using that $\|\alpha-\alpha_*\|_{C^3_-}<\epsilon_1$, we solve for $T(\gamma)$ and get $T(\gamma)<(1-C_1 \epsilon_1)^{-1}$. This implies that $T(\gamma)\leq T$ since  
\[
1-C_1\epsilon_1\geq 1-\frac{T-1}{2}\geq 1-\frac{T-1}{T}=\frac1T.
\]

We suppose that $\gamma\in\mathcal P_T(\alpha)$ and prove the estimates in item (i). The first inequality comes from \eqref{eq:period_1} using that $T(\gamma)< 2$ and $C_2\geq 2C_1$. For the second inequality, exploiting \eqref{eq:estimates2} and \eqref{eq:period_1} we have
\begin{align*}
|{\gamma}_{z,\rep}(s)-\bar\gamma_z(s)|&\leq |x_\gamma(sT(\gamma))|+|\wh\phi_\gamma(s T(\gamma))-s|\\
&\leq C_1\|\alpha-\alpha_*\|_{C^3_-}T(\gamma)+|\wh\phi_\gamma(sT(\gamma))-sT(\gamma)|+|T(\gamma)-1|s\\
&\leq C_1\|\alpha-\alpha_*\|_{C^3_-}T(\gamma)+C_1\|\alpha-\alpha_*\|_{C^3_-}T(\gamma)+|T(\gamma)-1|\\
& \leq 6C_1\|\alpha-\alpha_*\|_{C^3_-}.
\end{align*}
The higher derivatives can be bounded through \eqref{eq:estimates1} and  \eqref{eq:period_1}:
\begin{align*}
\Big\|\frac{\di \gamma_{z,\rep}}{\di s}-\frac{\di\bar\gamma_z}{\di s}\Big\|_{C^{2}}&\leq \big\|T(\gamma)(\dot x_\gamma,\dot\phi_\gamma)_\text{rep}-R_\st\big\|_{C^{2}}\\
&\leq T\big\|(\dot x_\gamma,\dot\phi_\gamma-1)_\text{rep}\big\|_{C^2}+|T(\gamma)-1|\\
&\leq T\cdot T^2\big\|(\dot x_\gamma,\dot\phi_\gamma-1)\big\|_{C^2}\!+2C_1\|\alpha-\alpha_*\|_{C^3_-}\\
&\leq 2^3 C_1\|\alpha-\alpha_*\|_{C^3_-}+2C_1\|\alpha-\alpha_*\|_{C^3_-}.
\end{align*}
Let us prove (ii). From \cite[Section III]{Gin87} or \cite[Section 3.2]{APB14}, up to shrinking $\epsilon_1$, for every $\alpha\in\mathcal B(\epsilon_1)$, there exists a differentiable function $S_\alpha:\Sigma\to\R$ with the following property. The set $\Crit S_\alpha$ is the union of the supports of the orbits $\gamma\in\mathcal P_T(\alpha)$. Therefore, $\mathcal P_T(\alpha)$ is non-empty as $\Crit S_\alpha$ is non-empty. The set $\mathcal P_T(\alpha)$ is also compact by the Arzel\`a-Ascoli theorem, as its elements have uniformly bounded period, and $\Sigma$ is compact. Finally, by item (i) we have $\mathcal P_T(\alpha)=\mathcal P_T(\alpha,\mathfrak h)$.
\end{proof}

\subsection{Bringing the Reeb flow to normal form}
In this subsection, we show that if $\alpha$ lies in $\mathcal B(\epsilon_1)$ and $\gamma\in\mathcal P_T(\alpha,\mathfrak h)$, we can suppose that $\gamma$ is a given flow line of $R_*$, up to rescaling $\alpha$ and applying a diffeomorphism of $\Sigma$. 

\begin{lem}\label{l:diffeo_equi}
There is a constant $C_3>0$ with the following property. For all $z_0,z_1\in \Sigma$, there exists an $S^1$-equivariant diffeomorphism $\Psi_{z_0,z_1}:\Sigma\rightarrow\Sigma$ isotopic to the identity with
\begin{equation*}
\bullet\quad \Psi_{z_0,z_1}(z_0)=z_1,\qquad\bullet\quad \Psi_{z_0,z_1}^*\alpha_*=\alpha_*,\qquad \bullet\quad \|\di\Psi_{z_0,z_1}\|_{C^2}\leq C_3,\ \ \|\di(\Psi_{z_0,z_1}^{-1})\|_{C^2}\leq C_3.
\end{equation*}
\end{lem}
\begin{proof}
We start with a local construction. Let $K_0:B\to [0,1]$ be a function which is equal to $1$ in a neighbourhood of $B'$ and whose support is contained in the interior of $B$. For every $(x',\phi')\in B'\times S^1$, let $\hat\phi'\in[0,1)$ be a lift of $\phi'$. We define
\[
K_{1}:B\to \R,\qquad K_{1}(x):=\hat\phi'+g_\st(x,i x').
\]
We let $K_{x'}:B\to\R$ be the function $K_{x'}:=K_0K_1$ and $\Phi^X_t$ the flow on $B\times S^1$ generated by the unique vector field $X$ such that
\[
\alpha_{\op{st}}(X)=K_{x'}\circ\mathfrak p_\st,\qquad \iota_X\di \alpha_{\op{st}}=-\di( K_{x'}\circ\mathfrak p_\st).
\]
Namely, $K_{x'}\circ\mathfrak p_\st$ is the contact Hamiltonian of $\Phi^X_t$ according to \cite[Section 2.3]{Gei08}. The vector field $X$ is compactly supported and an application of Moser's trick shows that
\begin{equation}\label{e:moserloc}
(\Phi^X_t)^*\alpha_{\op{st}}=\alpha_{\op{st}}, \qquad\forall\, t\in\R.
\end{equation}
The flow $\Phi^X_t$ lifts the Hamiltonian flow of the function $K_{x'}$ with respect to $\omega_{\mathrm{st}}$ on $B$. Moreover, since the curve $t\mapsto (tx',0)$ is $\alpha_\st$-Legendrian and $K_{x'}(tx')=\hat\phi'$, we see that
\[
\Phi^X_t(0,0)=(tx',t\wh\phi')\in B'\times S^1,\qquad \forall\,t\in[0,1].
\]
Then, the map $\Psi_{B,(x',\phi')}:=\Phi^X_1$ is a compactly supported diffeomorphism of $B\times S^1$ sending $(0,0)$ to $(x',\phi')$ and there exists a positive constant $C'$, independent of $(x',\phi')$, such that
\begin{equation}\label{e:diffeopsilocal}
\Vert \di\Psi_{B,(x',\phi')}\Vert_{C^2}\leq C', \ \ \Vert \di(\Psi_{B,(x',\phi')}^{-1})\Vert_{C^2}\leq C'.
\end{equation}
This completes the local construction. For the global argument, we observe that there exists $m\in\N^*$ independent of $z_0,z_1$ and a chain of points
\[
(z_u)\subset \Sigma,\qquad u\in U:=\big\{ ju_1\ \big|\ j=0,\cdots,m \big\},\qquad u_1:=1/m
\]
such that
\[
\forall\, u\in U\setminus\{1\}, \quad \exists\, y_u\in Z,\quad z_u,z_{u+u_1}\in \Sigma'_{y_u}.
\]
We construct $\Psi_{z_0,z_1}$ as the composition of $m$ maps $\Psi_{z_u,z_{u+u_1}}:\Sigma\to\Sigma$, $u\in U\setminus\{1\}$. Consider the trivialisation $\Df_{y_u}:B\times S^1\to \Sigma_{y_u}$ and define
\begin{align*}
\Psi_{z_u,z_{u+u_1}}:\Sigma\to \Sigma,\qquad\, \Psi_{z_u,z_{u+u_1}}:=\Df_{y_u}\circ\Big( \Psi_{B,\Df_{y_u}^{-1}(z_{u+u_1})}\circ \Psi_{B,\Df_{y_u}^{-1}(z_{u})}^{-1}\Big)\circ\Df_{y_u}^{-1}.
\end{align*}
The lemma follows from \eqref{e:moserloc} and \textbf{(DF3)} together with \eqref{e:diffeopsilocal} and the classical estimate on the $C^2$-norm of the differential of a composition of maps. In particular, the constant $C_3$ that we find depends only on $\Sigma$ and the Darboux family.
\end{proof}
\begin{dfn}\label{d:gamma*}
Let us fix a reference point $z_*\in Z$ with $q_*:=\mathfrak p(z_*)$ and define $\gamma_*:S^1\to\Sigma$ to be the prime periodic orbit of $R_*$ passing through $z_*$ at time $0$. We say that a contact form $\alpha$ is \textbf{normalised}, if $\gamma_*\in\mathcal P(\alpha)$. For every $\epsilon\in(0,\epsilon_0]$, we define the set
\begin{equation*}\label{e:alphanormal}
\mathcal B_*(\epsilon):=\big\{\,\alpha\in\mathcal B(\epsilon)\ \big|\ \alpha\text{ is normalised}\,\big\}.
\end{equation*} 
\end{dfn}
\begin{dfn}\label{d:rescale}
Let $c$ be a positive number and $\Psi:\Sigma\to\Sigma$ a diffeomorphism. For every contact form $\alpha$ on $\Sigma$, we write $\alpha_{c,\Psi}:=\frac{1}{c}\Psi^*\alpha$, so that $\Vol(\alpha)=c^2\Vol(\alpha_{c,\Psi})$ and we have a bijection 
\[
\begin{array}{rcl}
P(\alpha)&\longrightarrow&\mathcal P(\alpha_{c,\Psi}),\\ \gamma&\longmapsto&\gamma_{c,\Psi}
\end{array}\qquad \left\{\begin{aligned}
\gamma_{c,\Psi}(s):&=(\Psi^{-1}\circ\gamma)(cs),\ \forall\, s\in\R,\\ T(\gamma_{c,\Psi})&=\tfrac1c T(\gamma).
\end{aligned}\right.
\]
\end{dfn}
The next result is analogous to \cite[Proposition 3.10]{ABHS15}.
\begin{prp}\label{prp:normal}
Let $T$ be a number in $(1,2)$. For every $\epsilon_2\in(0,\epsilon_0]$, there is $\epsilon_3\in(0,\epsilon_0]$ (depending on $\epsilon_2$ and $T$) with the following properties. For all $\alpha\in\mathcal B(\epsilon_3)$ and all $\gamma\in\mathcal P_T(\alpha,\mathfrak h)$, there exists a diffeomorphism $\Psi:\Sigma\to \Sigma$ isotopic to the identity such that
\begin{equation*}
\alpha_{T(\gamma),\Psi}\in\mathcal B_*(\epsilon_2),\qquad \gamma_{T(\gamma),\Psi}=\gamma_*.
\end{equation*}
Moreover, the bijection $\mathcal P(\alpha)\to\mathcal P(\alpha_{T(\gamma),\Psi})$ restricts to a bijection $\mathcal P_T(\alpha,\mathfrak h)\to\mathcal P_T(\alpha_{T(\gamma),\Psi},\mathfrak h)$.
\end{prp}
\begin{proof}
Let $\alpha$ be an element of $\mathcal B(\epsilon_3)$ for some $\epsilon_3\leq\epsilon_1$ to be determined later on, and let $\gamma$ be a periodic orbit in $\mathcal P_T(\alpha,\mathfrak h)$. Here the constant $\epsilon_1$ is given by Proposition \ref{prp:gin}. We apply Lemma \ref{l:diffeo_equi} with $z_0=z_*$ and $z_1=\gamma(0)$ and get a diffeomorphism $\Psi_1:=\Psi_{z_*,\gamma(0)}:\Sigma\to\Sigma$ and a constant $C_3$ satisfying the properties described therein. We abbreviate $\alpha_1:=\Psi_{1}^*\alpha$. We get some $C'\geq 1$ depending on $C_3$ such that
\begin{equation}\label{e:alphaomega}
\Vert \alpha_1-\alpha_*\Vert_{C^3_-}\leq C'\Vert \alpha-\alpha_*\Vert_{C^3_-}.  
\end{equation}
The periodic curve $\gamma_1:=\Psi_{1}^{-1}\circ\gamma$ belongs to $\mathcal P_{T}(\alpha_1,\mathfrak h)$ and has period $T(\gamma)$. As $\gamma_1(0)=z_*$, we have $\bar\gamma_1=\gamma_*$. If $\epsilon_3\leq \tfrac{1}{C''}\epsilon_1$, then $\alpha_1\in\mathcal B(\epsilon_1)$ and Proposition \ref{prp:gin} implies that $\gamma_1\in\Sigma_{z_*}$ and
\begin{equation}\label{e:gammabar}
\|(\gamma_1)_{z_*,\rep}-(\gamma_*)_{z_*}\|_{C^3} \leq C''\Vert \alpha-\alpha_*\Vert_{C^3_-},\qquad C'':=C_2C'.
\end{equation}
We write $(\gamma_1)_{z_*,\rep}=(x_1,\phi_1)$ in the coordinates given by $\Df_{z_*}$. If $\epsilon_3$ is small enough, from \eqref{e:gammabar}, we see that $\Vert x_1\Vert_{C^0}<1/2$ and the map $\phi_1:S^1\rightarrow S^1$ is a diffeomorphism of degree $1$. In particular, there exists a unique map $\Delta\phi_1:S^1\to\R$, which lifts $\phi_1-\id_{S^1}$. We define a diffeomorphism $\Psi_{2,z_*}:B\times S^1\to B\times S^1$ by
\begin{equation*}
\Psi_{2,z_*}(x,s)=\Big(x+K\big(|x|\big)x_{1}(s),\ s+K\big(|x|\big)\Delta\phi_1(s)\Big),\quad \forall\,(x,s)\in B\x S^1,
\end{equation*}
where $K:[0,1]\rightarrow[0,1]$ is a function which is equal to $1$ on $[0,1/2]$ and equal to $0$ close to $1$. By \eqref{e:gammabar}, we have
\[
\Vert \Psi_{2,z_*}-\id_{B\times S^1}\Vert_{C^3}\leq C''\Vert K\Vert_{C^3}\Vert \alpha-\alpha_*\Vert_{C^3_-},
\]
which also implies
\begin{equation}\label{e:psiz2}
\Vert \di\Psi_{2,z_*}\Vert_{C^2}\leq 1+C''\Vert K\Vert_{C^3}\tfrac{1}{C'}\epsilon_1.
\end{equation}
Since $\Psi_{2,z_*}$ is compactly supported in the interior of $B\times S^1$, we can define $\Psi_2:\Sigma\to\Sigma$ as $\Psi_2:=\Df_{z_*}\circ \Psi_{2,z_*}\circ\Df^{-1}_{z_*}$ inside $\Sigma_{z_*}$ and as the identity in $\Sigma\setminus\Sigma_{z_*}$. We have $\Psi_2\circ\gamma_*=\gamma_{1,\rep}$, and thanks to \eqref{e:psiz2}, we see that $\Vert \di\Psi_2\Vert_{C^2}$ is bounded by a constant depending only on the Darboux family and on $C''\Vert K\Vert_{C^3}\tfrac{1}{C'}\epsilon_1$. Therefore, there is also a constant $C'''>0$ depending on the same quantities such that
\begin{equation}\label{e:psiz3}
\Vert\Psi_2^*(\alpha_1-\alpha_*)\Vert_{C^3_-}\leq C'''\Vert\alpha_1-\alpha_*\Vert_{C^3_-}.
\end{equation} 
We define
\[
\Psi:=\Psi_1\circ\Psi_2:\Sigma\to\Sigma,\qquad \epsilon_2':=\min\Big\{\epsilon_2,\epsilon_1,\frac{1}{C_2}\frac{T-1}{T+1}\Big\},
\]
and prove that $\alpha_{T(\gamma),\Psi}$ belongs to $\mathcal B_*(\epsilon'_2)$, provided $\epsilon_3$ is suitably small. We take
\[
\delta_0:=\frac{\epsilon_2'}{(T+1)C_{\mathfrak D}},
\]
and let $\delta_{1}>0$ be such that
\begin{equation}\label{e:psiz4}
\Vert \Psi_{2,z_*}-\id_{B\times S^1}\Vert_{C^3}\leq \delta_{1}\quad \Longrightarrow\quad \Vert (\Psi_{2,z_*})^*\alpha_\st-\alpha_\st\Vert_{C^3_-}\leq\delta_0.
\end{equation}
We assume further that
\[
\epsilon_3\leq \min\Big\{\frac{\delta_1}{C''\Vert K\Vert_{C^3}},\frac{1}{C_2}\frac{T-1}{T+1}\Big\}.
\]
This implies that $\Vert \Psi_{2,z_*}-\id_{B\times S^1}\Vert_{C^3}\leq \delta_{1}$ and we compute
\begin{align*}
\Vert \alpha_{T(\gamma),\Psi}-\alpha_*\Vert_{C^3_-}\leq &\left|\frac{1}{T(\gamma)}-1\right|\Vert \alpha_*\Vert_{C^3_-}+\frac{1}{T(\gamma)}\Vert \Psi^*\alpha-\alpha_*\Vert_{C^3_-}.
\end{align*}
For the first summand of the right-hand side, we first estimate $T(\gamma)^{-1}\leq \tfrac12(T+1)$ and then
\[
\left|\frac{1}{T(\gamma)}-1\right|\Vert \alpha_*\Vert_{C^3_-}\leq \frac{T+1}{2}C_2\epsilon_3\Vert \alpha_*\Vert_{C^3_-}.
\]
For the second summand, we estimate
\begin{align*}
\Vert \Psi^*\alpha-\alpha_*\Vert_{C^3_-}&\leq\Vert\Psi_2^*(\alpha_1-\alpha_*)\Vert_{C^3_-}+\Vert \Psi_2^*\alpha_*-\alpha_*\Vert_{C^3_-}\\
&\leq C'''\Vert\alpha_1-\alpha_*\Vert_{C^3_-}+C_{\mathfrak D}\Vert (\Psi_{2,z_*})^*\alpha_\st-\alpha_\st\Vert_{C^3_-}\\
&\leq C'C'''\epsilon_3+C_{\mathfrak D}\delta_0,
\end{align*}
where we used \eqref{e:constantcd}, \eqref{e:alphaomega}, \eqref{e:psiz3}, and \eqref{e:psiz4}.
Using the definition of $\delta_0$ and putting the computations together, we find that
\begin{equation*}
\Vert \alpha_{T(\gamma),\Psi}-\alpha_*\Vert_{C^3_-}\leq \frac{T+1}{2}\Big(C_2\Vert\alpha_*\Vert_{C^3_-}+C'C'''\Big)\epsilon_3+\tfrac12\epsilon'_2.
\end{equation*}
The quantity on the right is smaller than $\epsilon_2'\leq \epsilon_2$, if $\epsilon_3$ is small enough. Finally, we compute
\[
\gamma_{T(\gamma),\Psi}=\Psi^{-1}\circ \gamma_\rep=\Psi^{-1}_2\circ \Psi_1^{-1}\circ \gamma_{\rep}=\Psi_2^{-1}\circ\gamma_{1,\rep}=\gamma_*.
\]

Let us now deal with the second part of the statement. Let $\widetilde\gamma\mapsto\widetilde \gamma_{T(\gamma),\Psi}$ be the bijection between $\mathcal P(\alpha)$ and $\mathcal P(\alpha_{T(\gamma),\Psi})$ introduced in Definition \ref{d:rescale}. Let us assume that $T(\widetilde\gamma)\leq T$. Since $\epsilon_3\leq \epsilon_1$ we can use Proposition \ref{prp:gin}.(i), and from $C_2\epsilon_3\leq \tfrac{T-1}{T+1}$, we see that
\begin{align*}
T(\widetilde \gamma_{T(\gamma),\Psi})=\frac{T(\widetilde \gamma)}{T(\gamma)}\leq \frac{1+C_2\epsilon_3}{1-C_2\epsilon_3}\leq T.
\end{align*}
Assume, conversely, that $T(\widetilde\gamma_{T(\gamma),\Psi})\leq T$. Since $\epsilon_2'\leq\epsilon_1$, we can use Proposition \ref{prp:gin}.(i) and find that
\[
T(\widetilde\gamma)=T(\widetilde\gamma_{T(\gamma),\Psi})T(\gamma)\leq (1+C_2\epsilon'_2)(1+C_2\epsilon_3)\leq \frac{2T}{T+1}\frac{2T}{T+1}=T\frac{4T}{(T+1)^2}\leq T.\qedhere
\]
\end{proof}

\subsection{Preparing the surface of section}\label{ss:surface}

As in the previous subsection, let $z_*$ be a reference point on $\Sigma$ with $q_*:=\mathfrak p(z_*)$ and set
\[
\check M:=M\setminus \{q_*\}.
\]
Let $e\in H^2_\dR(M)$ be minus the real Euler class of $\mathfrak p$ and let us adopt the notation
\[
t_\Sigma:=\langle e,[M]\rangle>0,
\]
where, as observed in the introduction, $\langle e,[M]\rangle=|H_1^\mathrm{tor}(\Sigma;\Z)|$. We define the annulus
\[
\mathbb A:=[0,a)\times S^1.
\]
We consider the inclusion $\mathfrak i_1:\mathring \A\to \A$, where $\mathring{\A}=(0,a)\times S^1$, and the map
\[
\mathfrak i_2:\mathring\A\to \check M,\qquad \mathfrak i_2(r,\theta)=\df_{z_*}(re^{2\pi i\theta}),
\]
where we identify the domain of $\mathfrak d_{z_*}$ with a subset of the complex plane. We glue together $\mathbb A$ and $\check M$ along the maps $\mathfrak i_1$ and $\mathfrak i_2$ to get a smooth compact surface $N$ with the same genus as $M$ and one boundary component denoted by $\p N$. Namely, we have the following commutative diagram
\[
\xymatrix{
\mathring\A \ar[r]^-{\mathfrak i_2} \ar[d]_{\mathfrak i_1}&\check M\ar[d]\\
\mathbb A\ar[r]& N}
\]
so that $\check M$ is diffeomorphic to the interior $\mathring N=N\setminus\p N$ and $\mathbb A$ to a collar neighbourhood of $\p N$. On $\check M$ we have the orientation given by $\om_*$, while on $\A$ the one given by $\di r\wedge\di\theta$. These two orientations glue together to an orientation of $N$, since $\mathfrak i_1$ and $\mathfrak i_2$ are orientation preserving. Using the usual convention of putting the outward normal first, we see that the orientation induced on $\p N$ is given by $-\di\theta$. As for $M$ and $\Sigma$, we fix on $N$ some auxiliary Riemannian metric to compute norms of sections, and distances between points and between diffeomorphisms. In particular, we write the $C^1$-distance on the space of diffeomorphisms from $N$ to itself as
\begin{equation*}\label{e:distc1}
\dist_{C^1}:\mathrm{Diff}(N)\times\mathrm{Diff}(N)\to\R.
\end{equation*}

Consider now the map
\[
S_{\mathbb A}:\mathbb A\to \Sigma,\qquad S_\A(r,\theta)=\Df_{z_*}(re^{2\pi i\theta},-t_\Sigma\theta)
\]
and observe that, for all $\theta\in S^1$, there holds $S_\A(0,\theta)=\gamma_*(-t_\Sigma\theta)$, so that
\begin{equation}\label{e:dsaboundary}
\di_{(0,\theta)}S_\A\cdot\p_\theta=-t_\Sigma R_*.
\end{equation}
The map $S_{\mathbb A}\circ \mathfrak i_2^{-1}: \mathfrak i_2(\mathring\A)\to \Sigma$ is a local section of the bundle $\mathfrak p$ with a singularity of order $-t_\Sigma$ at $q_*$. Since $-t_\Sigma$ is the Euler number of $\mathfrak p$, this section extends to a section on $\check M$ and yields a map $S_{\check M}:\check M\to\Sigma$. By the commutativity of the diagram above, we get a map $S:N\to \Sigma$ fitting into the diagram
\begin{equation*}
\xymatrix{&\mathbb A\ar[ld]\ar[rd]^{S_{\mathbb A}}&\\
N\ar[rr]^S& &\Sigma\,.\\ &\check M\ar[lu]\ar[ru]_{S_{\check M}}&}
\end{equation*}
Moreover, $S_{\check M}^*\di\alpha_*=(\mathfrak p\circ S_{\check M})^*\om_* = \om_*$, and $S_\A^*\di\alpha_*=r\di r\wedge\di \theta$. In particular, $S^*\di\alpha_*$ is a two-form on $N$, which is symplectic on the interior of $N$ and vanishes of order $1$ at the boundary of $N$. The one-form
\[
\lambda_*:=S^*\alpha_*
\]
is a primitive for $S^*\di\alpha_*$ such that
\begin{equation*}
\lambda_*|_\A=(\Df_{z_*}^{-1}\circ S_\A)^*(\di\phi+\mathfrak p_\st^*\lambda_\st)=\big(-t_\Sigma+\tfrac12r^2\big)\di\theta.
\end{equation*}
If $\alpha$ is a normalised form, so that $R_\alpha=R_*$ on $\mathfrak p^{-1}(q_*)$, we set
\begin{equation*}\label{e:deflambda1}
\lambda:=S^*\alpha,
\end{equation*}
and by equation \eqref{e:dsaboundary}, we have
\begin{equation}\label{e:boundaryproperties}
\bullet\ \ \lambda\big|_{\ta(\p N)}=\lambda_*\big|_{\ta(\p N)},\qquad \bullet \ \ \di\lambda=0\ \  \text{at\ }\p N.
\end{equation}
\begin{prp}\label{p:nub}
For all $\epsilon_4>0$, there exists a number $\epsilon_5\in(0,\epsilon_0]$ such that, if $\alpha\in\mathcal B_*(\epsilon_5)$, there exist a map $\zeta:N\to N$ isotopic to the identity and a function $b:N\to\R$ satisfying the following properties.
\begin{align*}
(i)&\ \ \text{Triviality at the boundary:}&&\zeta|_{\p N}=\id_{\p N},\quad b|_{\p N}=0,\\[1ex]
(ii)&\ \ \text{$C^1$-smallness:}&&\max\big\{\dist_{C^1}(\zeta,\id_N),\Vert b\Vert_{C^1}\big\}<\epsilon_4,\\[1ex]
(iii)&\ \ \text{Uniformisation:}&&\zeta^*\lambda-\lambda_*=\di b.
\end{align*}
\end{prp}
\begin{proof}
Let $\alpha\in\mathcal B_*(\epsilon_5)$, for some $\epsilon_5\in(0,\epsilon_0]$ to be determined. For all $u\in[0,1]$, we define $\lambda_u:=\lambda_*+u(\lambda-\lambda_*)$. On $\A$, we get
\[
\lambda-\lambda_*=c_1\di r+c_2\di\theta,\qquad \di\lambda_*=r\di r\wedge\di\theta,\qquad\di\lambda=f\di r\wedge\di\theta,\qquad \di\lambda_u=(r+u(f-r))\di r\wedge\di\theta,
\]
for some functions $c_1,c_2,f:\A\to\R$. By \eqref{e:boundaryproperties}, we have $c_2(0,\theta)=0$ and $f(0,\theta)=0$. Define the auxiliary function
\[
c_3:\A\to\R,\qquad c_3(r,\theta):=c_2(r,\theta)-\int_0^r\p_\theta c_1(r',\theta)\di r'.
\]
From the definition of $c_1,c_2,c_3$ and $f$, we have the chain of identities
\[
(\p_r c_3)\di r\wedge\di\theta=(\p_r c_2-\p_\theta c_1)\di r\wedge\di\theta=\di(\lambda-\lambda_*)=\di\lambda-\di\lambda_*=(f-r)\di r\wedge\di\theta,
\]
which implies
\begin{equation*}
\p_r c_3=f-r.
\end{equation*}
As a result, $c_3(0,\theta)=0, \p_r c_3(0,\theta)=0$ and there exists a function $\hat c_3:\A\to \R$ with $c_3=r\hat c_3$, $\hat c_3|_{\p N}=0$, and a function $\hat f:\A\to\R$ with $f=r\hat f$, defined by
\begin{align*}
\hat c_3(r,\theta):=\int_0^1\p_rc_3(vr,\theta)\di v=\int_0^1\big(f(vr,\theta)-vr\big)\di v,\qquad \hat f(r,\theta):=\int_0^1\p_rf(vr,\theta)\di v.
\end{align*}
In particular,
\begin{equation}\label{e:omegau}
\di\lambda_u=r\big(1+u(\hat f-1)\big)\di r\wedge \di\theta
\end{equation}
and we have the estimate
\begin{equation}\label{e:estc3fhat}
\max\big\{\Vert \hat c_3\Vert_{C^2},\Vert \hat f-1\Vert_{C^1}\big\}\leq \Vert f-r\Vert_{C^2}.
\end{equation}

We now look for paths $u\mapsto\zeta_u$ and $u\mapsto b_u$ with $\zeta_0=\id_N$ and $b_0=0$ such that
\begin{equation*}
\zeta_u^*\lambda_u-\di b_u=\lambda_*
\end{equation*}
so that, for $u=1$, we get a solution to item (iii) in the statement.
Let $X_u$ denote the vector field generating $\zeta_u$ and set $a_u:=\tfrac{\di}{\di u}b_u$. By differentiating the equation above with respect to $u$, we find that such an equation can be solved for $\zeta_u$ and $b_u$ if and only if
\begin{equation*}
(\lambda-\lambda_*)+\iota_{X_u}\di\lambda_u+\di\big(\lambda_u(X_u)\big)-\di (a_u\circ\zeta_u^{-1})=0.
\end{equation*}
Introducing an auxiliary function $h:N\to\R$, we see that $(X_u,a_u)$ is a solution if and only if
\begin{equation}\label{e:mosertwo}
\left\{\begin{aligned}\iota_{X_u}\di\lambda_u&=-(\lambda-\lambda_*)+\di h,\\
a_u&=\big(\lambda_u(X_u)+h\big)\circ \zeta_u.\end{aligned}\right.
\end{equation}
We define 
\[
\A':=[0,a/2)\times S^1
\]
and choose $h:=h_{\A}\cdot K$, where $K:N\to[0,1]$ is a bump function which is equal to $1$ on $\A'$ and its support is contained in $\A$, and $h_{\A}:\A\to \R$ is defined by
\[
h_{\A}(r,\theta):=\int_0^rc_1(r',\theta)\di r'.
\]
This function has the crucial property that
\begin{equation}\label{e:lambdalambdastar}
\lambda-\lambda_*-\di h= c_1\di r+c_2\di\theta-c_1\di r-\Big(\int_0^r\p_\theta c_1(r',\theta)\di r'\Big)\di\theta=r\hat c_3\di\theta\qquad \text{on}\ \ \A',
\end{equation}
which implies that the first equation in \eqref{e:mosertwo} admits a smooth solution $X_u$. Indeed, on the annulus $\A'$, we divide both sides of the equation by $r$ and using \eqref{e:omegau},  \eqref{e:lambdalambdastar}, we get
\[
X_u=-\frac{\hat c_3}{1+u(\hat f-1)}\partial_r.
\]
On $N\setminus \A'$, $X_u$ is uniquely determined by the fact that $\di\lambda_u|_{N\setminus \A'}$ is symplectic. The vector $X_u$ vanishes at $\p N$, since $\hat c_3$ vanishes there, as observed before. This shows that $\zeta_u$ is the identity at the boundary.
If we choose $\epsilon_5$ small, we see that the $C^2$-norm of $c_1$, $c_2$ and $f-r$ are small, and consequently, also the $C^2$-norm of $h$.
By \eqref{e:estc3fhat}, we conclude that the $C^1$-norm of $X_u$ is small, as well. As a consequence, also $\dist_{C^1}(\zeta_u,\id_N)$ is small. Therefore, by defining $\zeta:=\zeta_1$ and taking $\epsilon_5$ small enough, we get $\dist_{C^1}(\zeta,\id_N)<\epsilon_4$. We can now define $a_u$ through the second equation in \eqref{e:mosertwo}. From the estimates on $\lambda,X_u,\zeta_u$ and $h$, we see that, if $\epsilon_5$ is small, $\Vert a_u\Vert_{C^1}<\epsilon_4$ and the same is true for $b:=b_1$. Since $h$ and $X_u$ vanish at the boundary, we also have $a_u|_{\p N}=0$, and as $b_0|_{\p N}=0$, the function $b$ vanishes at the boundary, as well.
\end{proof}
\subsection{The open book decomposition and the first return map}
Combining the map $S$ with the Reeb flow of $\alpha_*$, we get a rational open book for $\Sigma$:
\begin{equation*}\label{e:xins1}
\begin{aligned}
\Xi:N\x S^1&\longrightarrow\Sigma\\
(q,s)&\longmapsto \Phi^{\alpha_*}_s\big(S(q)).
\end{aligned}
\end{equation*}
If $\mathfrak i_N:N\into N\x S^1$ is the canonical embedding $\mathfrak i_N(x)=(x,0)$, then $S=\Xi\circ\mathfrak i_N$. On the collar neighbourhood $\A\times S^1$ of $\p (N\times S^1)$, $\Xi$ has the coordinate expression
\begin{equation}\label{e:xia}
\begin{aligned}
\Xi_\A:\A\x S^1&\longrightarrow B\x S^1\\
\big((r,\theta),s\big)&\longmapsto \big(r e^{2\pi i\theta},s-t_\Sigma\theta\big).
\end{aligned}
\end{equation}
The restricted map $\mathring \Xi:\mathring N\times S^1\to \mathfrak p^{-1}(\check M)=\Sigma\setminus\mathfrak p^{-1}(q_*)$ is a diffeomorphism, and $\mathring{\Xi}^*\alpha_*$ is a contact form on $\mathring N\x S^1$ with Reeb vector field $R_{\mathring{\Xi}^*\alpha_*}=\p_s$, which smoothly extends to the whole $N\times S^1$. If we write ${\mathfrak i}_{\p N\times S^1}:\p N\times S^1\to N\times S^1$ for the standard embedding of the boundary, the map $\Xi\circ{\mathfrak i}_{\p N\times S^1}:\p N\times S^1\to \mathfrak p^{-1}(q_*)\subset\Sigma$ has the coordinate expression $(\theta,s)\mapsto (s-t_\Sigma\theta)$. Therefore we have
\begin{equation}\label{e:xiboundary}
\di(\Xi\circ {\mathfrak i}_{\p N\times S^1})\cdot \partial_\theta=-t_\Sigma R_*,\qquad \di(\Xi\circ{\mathfrak i}_{\p N\times S^1})\cdot\partial_s=R_*.
\end{equation}
If $\alpha$ is a normalised contact form, we define the pull-back form
\begin{equation*}\label{def:betaxialpha}
\beta:=\Xi^*\alpha.
\end{equation*}
The next result is the analogue of \cite[Proposition 3.6]{ABHS15}.
\begin{prp}\label{prop:vector_field_openbook}
If $\alpha$ is a normalised contact form, then
\begin{enumerate}[(i)]
\item There hold
\begin{equation*}
{\mathfrak i}_{\p N\times S^1}^*\beta=\di s-t_\Sigma \di\theta,\qquad\quad \di\beta|_{\p N\times S^1}=0.
\end{equation*}
By the latter identity we mean that $\di\beta_z(\xi)=0$ for all $z\in\p N\x S^1$ and $\xi\in T_z(N\x S^1)$.
\item The Reeb vector field $R_{\mathring\Xi^*\alpha}$ of $\mathring\Xi^*\alpha$ on $\mathring N\x S^1$ smoothly extends to a vector field $R_{\beta}$ on the whole $N\x S^1$, so that, at every point in $\p(N\times S^1)$, $R_{\beta}$ is tangent to $\p (N\times S^1)$.
\item If we denote by $\Phi^\beta$ the flow of $R_\beta$, we have
\begin{equation*}
\beta(R_{\beta})=1,\qquad \iota_{R_{\beta}}\di\beta=0,\qquad (\Phi^{\beta}_t)^*\beta=\beta,\quad \forall\, t\in\R.
\end{equation*}
\item For every $\epsilon_6>0$, there exists $\epsilon_7\in(0,\epsilon_0]$, independent of $\alpha$, such that
\begin{equation*}
\alpha\in\mathcal B_*(\epsilon_7)\quad\Longrightarrow\quad \Vert R_{\beta}-\p_s\Vert_{C^1}<\epsilon_6.
\end{equation*}
\end{enumerate}
\end{prp}
\begin{proof}
By \eqref{e:xiboundary}, equation $\alpha(R_\alpha)=1$, and the fact that $R_\alpha=R_*$ on $\mathfrak p^{-1}(q_*)$ as $\alpha$ is normalised, we get the first equality in item (i). Since $\p N\times S^1$ has co-dimension $1$ in $N\times S^1$, to prove the second equality it is enough to show that for all vectors $v\in \ta (\p N\times S^1)$, we have $\iota_v\di\beta=0$. As $v$ is a linear combination of $\p_\theta$ and $\p_s$, this follows again from \eqref{e:xiboundary} and the fact that $R_\alpha$ annihilates $\di\alpha$.

Now we prove (ii). We set $\alpha_{z_*}:=\mathfrak D_{z_*}^*\alpha$, which is a contact form on $B\times S^1$ with corresponding Reeb vector field $R_{z_*}$. Using coordinates $(x,\phi)\in B\times S^1$, we have the splitting 
\[
R_{z_*}(x,\phi)=R^x_{z_*}(x,\phi)+R^\phi_{z_*}(x,\phi)\p_\phi.
\] 
Since $R_\alpha$ is tangent to $\mathfrak p^{-1}(q_*)$, there holds $R^x_{z_*}(0,\phi)=0$, and therefore, there exists a matrix-valued function $W_{z_*}$ such that 
\begin{equation*}
R^x_{z_*}(x,\phi) = W_{z_*}(x,\phi)\cdot x,\qquad \Vert W_{z_*}\Vert_{C^1}\leq \Vert R^x_{z_*}\Vert_{C^2},
\end{equation*}
by Lemma \ref{l:divi}. We can then write $R^x_{z_*}$ in polar coordinates on $(B\setminus\{0\})\times S^1$ as
\[
R^x_{z_*}(re^{2\pi i\theta},\phi)=g_\st\Big( W_{z_*}(re^{2\pi i\theta},\phi)\cdot re^{2\pi i\theta},e^{2\pi i\theta}\Big)\p_ r+ g_\st\Big(W_{z_*}(re^{2\pi i\theta},\phi)\cdot re^{2\pi i\theta}, \frac{ie^{2\pi i\theta}}{r}\Big)\p_\theta.
\]
In particular, if we set
\begin{equation*}
\left\{\begin{aligned}
R^r_{z_*}(r,\theta,\phi)&:= g_\st\Big(W_{z_*}(re^{2\pi i\theta},\phi)\cdot re^{2\pi i\theta},e^{2\pi i\theta}\Big),\\
R^\theta_{z_*}(r,\theta,\phi)&:= g_\st\Big( W_{z_*}(re^{2\pi i\theta},\phi)\cdot e^{2\pi i\theta},i e^{2\pi i\theta}\Big),
\end{aligned}\right.
\end{equation*}
then $R^r_{z_*}\circ \Xi_\A$ and $R^\theta_{z_*}\circ \Xi_\A$ are smooth functions on $\A\x S^1\subset N\x S^1$ with
\begin{equation}\label{e:rrthetax}
\max\Big\{\Vert R^r_{z_*}\circ \Xi_\A\Vert_{C^1},\Vert R^\theta_{z_*}\circ \Xi_\A\Vert_{C^1}\Big\}\leq (1+\Vert\di\Xi_\A\Vert_{C^0})\Vert R^x_{z_*}\Vert_{C^2}.
\end{equation}
Differentiating formula \eqref{e:xia}, we get 
\[
\di\Xi_\A\cdot\p_r=\p_r,\quad \di\Xi_\A\cdot\p_\theta=\p_\theta-t_\Sigma\p_\phi,\qquad\di\Xi_\A\cdot\p_s=\p_\phi.
\]
Thus, we conclude that 
\begin{equation*}
R_{\beta}:=(R^r_{z_*}\circ \Xi_\A)\p_r+(R_{z_*}^\theta\circ\Xi_\A)\p_\theta+\big(t_\Sigma R_{z_*}^\theta\circ \Xi_\A+R_{z_*}^\phi\circ \Xi_\A\big)\p_s
\end{equation*}
is the desired extension of $R_{\mathring\Xi^*\alpha}$ in the collar neighbourhood $\A\x S^1$ of $\p N\x S^1$. As $R^r_{z_*}\circ \Xi_\A$ vanishes at $r=0$, the extended vector field is tangent to $\p N\times S^1$ and (ii) is proven.\\[-2ex]

By the very definition of the Reeb vector field, $(\mathring\Xi^*\alpha)(R_{\mathring\Xi^*\alpha})=1$ and $\iota_{R_{\mathring\Xi^*\alpha}}\di(\mathring\Xi^*\alpha)=0$. By continuity of the extended vector field $R_{\beta}$, the first two relations in item (iii) follow. The third one is a consequence of the first two and Cartan's formula. Point (iii) is established.  \\[-2ex]

We assume that $\alpha\in\mathcal B_*(\epsilon_7)$, for some $\epsilon_7$ to be determined independently of $\alpha$ and we prove (iv) by estimating $\Vert R_{\beta}-\p_s\Vert_{C^1}$ separately on $\overline{N\setminus\A}\times S^1$ and $\A\times S^1$. Since $(\overline{N\setminus\A})\times S^1$ is compact, there exists a constant $C'>0$ depending on $\Vert \di\mathring{\Xi}\Vert_{C^2}$ but not on $\alpha$ such that 
\begin{align*}
\Vert \mathring\Xi^*\alpha-\mathring \Xi^*\alpha_*\Vert_{C^3_-}&\leq C'\Vert \alpha- \alpha_*\Vert_{C^3_-}.
\end{align*}
Therefore, as in \eqref{e:ralphaz}, $\Vert R_{\beta}-\p_s\Vert_{C^2}$ is smaller than $\epsilon_6$ on $(\overline{N\setminus\A})\times S^1$, if $\epsilon_7$ is small enough. On $\A\times S^1$, there is some $C''>0$ for which we have the inequality 
\begin{align*}
\Vert R_{\beta}-\p_s\Vert_{C^1}&\leq C''\max\Big\{\Vert R^r_{z_*}\circ\Xi_\A\Vert_{C^1},\Vert R^\theta_{z_*}\circ\Xi_\A\Vert_{C^1},\Vert R^\phi_{z_*}\circ\Xi_\A-1\Vert_{C^1}\Big\}\\
&\leq C''(1+\Vert \di \Xi_\A\Vert_{C^0})\Vert R_{z_*}-R_{\st}\Vert_{C^2}\\
&\leq C''(1+\Vert \di \Xi_\A\Vert_{C^0})A_2C_{\mathfrak D}\Vert \alpha-\alpha_*\Vert_{C^3_-},
\end{align*}
where we have used \eqref{e:rrthetax}, the equality $R_{\st}=\p_\phi$ and inequality \eqref{e:inrz}. This proves that $\Vert R_\beta-\p_s\Vert_{C^1}$ is smaller than $\epsilon_6$ on $\A\times S^1$, if $\epsilon_7$ is small enough.
\end{proof}

We can now show that $S$ is a global surface of section for $\Phi^\alpha$ with certain properties. 
\begin{dfn}
Let $\Phi$ be a flow on $\Sigma$ without rest points and $N_1$ a compact surface. A map $S_1:N_1\to \Sigma$ is a \textbf{global surface of section} for $\Phi$ if the following properties hold:
\begin{itemize}
\item The map $S_1|_{\mathring{N}_1}$ is an embedding and the map $S_1|_{\p N_1}$ is a finite cover onto its image;
\item The surface $S_1(\mathring N_1)$ is transverse to the flow $\Phi$ and $S_1(\p N_1)$ is the support of a finite collection of periodic orbits of $\Phi$;
\item For each $z\in\Sigma\setminus S_1(\p N_1)$, there are $t_-<0<t_+$ such that $\Phi_{t_-}(z)$, $\Phi_{t_+}(z)$ lie in $S_1(\mathring N_1)$.\end{itemize}
\end{dfn}
Before stating the proposition, we introduce the following notation. Let $q\in \p N\cong S^1$ and denote by $-q\in\p N$ its antipodal point. By 
\begin{equation}\label{e:lambdaxy}
\int_{q}^{q'}\lambda_*,\qquad q'\in\p N\setminus\{-q\},
\end{equation}
we mean the integral of $\lambda_*$ over any path connecting $q$ and $q'$ within $\p N\setminus\{-q\}$. This number does not depend on the choice of such path.

\begin{prp}\label{p:return}
Let $T$ be a real number in the interval $(1,2)$. For all $\epsilon_8>0$, there exists $\epsilon_9\in(0,\epsilon_0]$ with the following properties. If $\alpha\in\mathcal B_*(\epsilon_9)$, then $S:N\to \Sigma$ is a global surface of section for $\Phi^\alpha$ with the first return time admitting an extension to the boundary
\begin{equation*}
\tau:N\to\R,\qquad \tau(q):=\inf\big\{t>0\ \big|\ \Phi^{\beta}_{t}(q,0)\in N\times\{0\}\big\}
\end{equation*}
and the first return map admitting an extension to the boundary
\begin{equation*}
P:N\to N,\qquad (P(q),0):=\Phi^{\beta}_{\tau(q)}(q,0).
\end{equation*}
Moreover, the following properties hold.
\begin{align*}
(i)&\ \ \text{$C^1$-smallness:}&&\max\big\{\dist_{C^1}(P,\id_N),\Vert \tau-1\Vert_{C^1}\big\}<\epsilon_8,\\[1ex]
(ii)&\ \ \text{Normalisation:}&&\tau(q)=1+\int_{q}^{P(q)}\lambda_* ,\quad\forall\, q\in\p N,\\[1ex]
(iii)&\ \ \text{Exactness:}&&P^*\lambda=\lambda+\di \tau,\\
(iv)&\ \ \text{Volume:}&&\Vol(\alpha)=\int_N\tau\,\di\lambda,\\
(v)&\ \ \text{Fixed points:}&&q\in\mathring N\ \Longrightarrow\ \Big[\,q\in\Fix(P)\ \Longleftrightarrow\ \gamma_q(t):=\Phi^\alpha_t\big(S(q)\big)\in\mathcal P_T(\alpha,\mathfrak h)\,\Big],\\[1ex]
(vi)&\ \ \text{Period:}&& q\in\mathring N\cap\Fix(P)\ \ \Longrightarrow\ \ T(\gamma_q)=\tau(q),\\[1ex]
(vii)&\ \ \text{Zoll case:}&& \text{if $\alpha$ is Zoll, then $P= \id_N$}.
\end{align*}
\end{prp}
\begin{proof}
Let $\epsilon_6\in(0,1)$, which we will take small enough depending on $\epsilon_8$, and let $\epsilon_7$ be the number associated with $\epsilon_6$ in Proposition \ref{prop:vector_field_openbook}. If $\alpha\in\mathcal B_*(\epsilon_7)$, then $1-\epsilon_6<\di s(R_{\beta})<1+\epsilon_6$.
This implies at once that $S:N\to\Sigma$ is a global surface of section. In particular, if $q\in N$, there exists a smallest positive time $\tau(q)$ such that $\Phi^{\beta}_{\tau(q)}(q,0)$ belongs to $N\times \{0\}$. We estimate the return time more precisely as $(1+\epsilon_6)^{-1}<\tau(q)<(1-\epsilon_6)^{-1}$. In particular, if $\epsilon_6$ is small enough, there holds
\begin{equation}\label{e:minmaxtau}
\max\tau<T<2\min \tau.
\end{equation}
Shrinking $\epsilon_6$ further, if necessary, we also get $\Vert \tau-1\Vert_{C^1}<\epsilon_8$ from Proposition \ref{prop:vector_field_openbook}.(iv).

We define the return point $P(q)$ by the equation $(P(q),0)=\Phi^{\beta}_{\tau(q)}(q,0)$. Again by Proposition \ref{prop:vector_field_openbook}.(iv), we can achieve $\dist_{C^1}(P,\id_N)<\epsilon_8$, if $\epsilon_6$ is small enough, so that item (i) is established. Let $q\in\mathring N$ and let us prove the statement in square brackets in item (v). If $q\in \Fix(P)$, then $\gamma_q$ is prime, since intersects $S(N)$ only once, and has period $\tau(q)$. By \eqref{e:minmaxtau}, we have $\tau(q)<T$. Hence, by Proposition \ref{prp:gin}.(ii) the curve $\gamma_q$ belongs to $\mathcal P_T(\alpha,\mathfrak h)$. Suppose conversely that $\gamma_q$ has period $T(\gamma_q)\leq T$. If $q\neq P(q)$, then we would get the contradiction
\[
T(\gamma_q)\geq \tau(q)+\tau(P(q))\geq 2\min \tau>T.
\]
This establishes item (v) and (vi), at once. Let us assume that $\alpha$ is Zoll. Since $\gamma_*\in\mathcal P(\alpha,\mathfrak h)$, then, if $q\in\mathring N$, the orbit $\gamma_q$ belongs to $\mathcal P(\alpha,\mathfrak h)$ and satisfies
\[
T(\gamma_q)=T(\gamma_*)=1<T.
\]
By item (v), we conclude that $q\in\Fix(P)$. This shows $\mathring N\subset\Fix(P)$, and by continuity $\Fix(P)=N$. Namely, $P=\id_N$ and item (vii) holds.

Let $q\in\p N$ and denote by $\delta_q:[0,\tau(q)]\to \p N\times S^1$ the curve $\delta_q(t)=\Phi^{\beta}_t(q,0)$. Using coordinates $(\theta,s)$ on $\p N\times S^1$, we can write $\delta_q(s)=(\theta_q(t),s_q(t))$, so that $\theta_q:[0,\tau(q)]\to\p N$ is a path between $\theta_q(0)=q$ and $\theta_q(\tau(q))=P(q)$, and $s_q(0)=0, s_q(\tau(q))=1$. We compute
\begin{align*}
\tau(q)=\int_0^{\tau(q)}\di t=\int_0^{\tau(q)}\delta_q^*({\mathfrak i}_{\p N\times S^1}^*\beta)=\int_0^{\tau(q)}\delta_q^*\big(\di s-t_\Sigma\di\theta\big)&=\int_0^{\tau(q)}\big(\di s_q+\theta_q^*(-t_\Sigma\di\theta)\big)\\
&=1+\int_0^{\tau(q)}\theta_q^*\lambda_*
\end{align*}
and the integral of $\lambda_*$ over $\theta_q$ is equal to $\int_{q}^{P(q)}\lambda_*$, as $\theta_q$ is short if $\epsilon_6$ is small enough. This establishes item (ii). Therefore, we can choose $\epsilon_9:=\epsilon_7$ in the statement of the corollary.

We prove now item (iii) and (iv) by considering the map
\[
Q:[0,1]\times N\to N\times S^1,\qquad Q(t,q):=\Phi^{\beta}_{t\,\tau(q)}({\mathfrak i}_N(q)),
\]
where $\mathfrak i_N:N\into N\x S^1$ is the canonical embedding. Its differential is given by
\[
\di_{(t,q)} Q=\di_{(t,q)}(t\,\tau)\otimes R_{\beta}(Q(t,q))+\di_{{\mathfrak i}_N(q)}\Phi^{\beta}_{t\, \tau(q)}\cdot\di_q{\mathfrak i}_N. 
\]
Hence, using Proposition \ref{prop:vector_field_openbook}.(iii) we compute
\begin{align*}
Q^*\beta=\beta_{Q}\Big(\di(t\,\tau)\otimes R_{\beta}(Q)+\di_{{\mathfrak i}_N}\Phi^{\beta}_{t\,\tau}\cdot\di{\mathfrak i}_N\Big)&=\di(t\,\tau) \beta(R_\beta) +(\Phi^{\beta}_{t\,\tau}\circ {\mathfrak i}_N)^*\beta\\
&=\di(t\,\tau)+{\mathfrak i}_N^*(\Phi^{\beta}_{t\,\tau})^*\beta\\
&=\di(t\,\tau)+{\mathfrak i}_N^*\beta\\
&=\di(t\,\tau)+\lambda.
\end{align*}
We define ${\mathfrak i}_1:N\to [0,1]\times N$ by ${\mathfrak i}_1(q)=(1,q)$ and observe that $Q\circ{\mathfrak i}_1={\mathfrak i}_N\circ P$. Therefore, 
\[
P^*\lambda=P^*{\mathfrak i}_N^*\beta=({\mathfrak i}_N\circ P)^*\beta=(Q\circ {\mathfrak i}_1)^*\beta={\mathfrak i}_1^*Q^*\beta={\mathfrak i}_1^*\big(\di(t\,\tau)+\lambda\big)=1\di\tau+\lambda.
\]
This establishes item (iii). We calculate the volume of $\alpha$  pulling back by $\Xi\circ Q$:
\begin{align*}
\Vol(\alpha)=\int_{N\times S^1}\beta\wedge\di\beta=\int_{[0,1]\times N}\big(\di(t\,\tau)+\lambda\big)\wedge\di\lambda&=\int_{[0,1]\x N}\di(t\,\tau)\wedge\di\lambda\\&=\int_{[0,1]\x N}\di\big(t\,\tau\,\di\lambda\big)\\
&=\int_N1\,\tau\,\di\lambda-\int_N0\,\tau\,\di\lambda,
\end{align*}
which yields item (iv).
\end{proof}
\subsection{Reduction to a two-dimensional problem}
Putting together all the results of this section, we are able to translate the systolic-diastolic inequality into a statement for maps on $N$. We recall the set-up. Let $\alpha_*$ be a Zoll contact form on a closed three-manifold $\Sigma$ with associated bundle $\mathfrak p:\Sigma\to M$. Let  $S:N\to \Sigma$ be a global surface of section for the Reeb flow $\Phi^\alpha$ of $\alpha\in\mathcal B_*(\epsilon_9)$ as described at the beginning of Section \ref{ss:surface} and in Proposition \ref{p:return}. Let $\lambda_*=S^*\alpha_*$ and remember that $t_\Sigma=\langle e,[M]\rangle$. The next result is the analogous of \cite[Lemma 3.7 \& Proposition 3.8]{ABHS15}

\begin{thm}\label{t:final3dim}
For any $T\in(1,2)$ and $\epsilon_{10}>0$, there is $\epsilon_{11}>0$ such that for all contact forms $\alpha'$ with $\Vert\di\alpha'-\di\alpha_*\Vert_{C^2}< \epsilon_{11}$, the set $\mathcal P_T(\alpha',\mathfrak h)$ is compact and non-empty. Moreover for every $\gamma\in\mathcal P_T(\alpha',\mathfrak h)$, there exist a diffeomorphism $\varphi:N\to N$, and a function $\sigma:N\to\R$ with the following properties.
\begin{align*}
(i)&\ \ \text{$C^1$-smallness:}&& d(\varphi,\id_N)_{C^1}<\epsilon_{10},\\[1ex]
(ii)&\ \ \text{Normalisation:}&& \sigma(q)=\int_{q}^{\varphi(q)}\lambda_*,\quad \forall q\in\p N.\\[1ex]
(iii)&\ \ \text{Exactness:}&&\varphi^*\lambda_*=\lambda_*+\di \sigma,\\[1ex]
(iv)&\ \ \text{Volume:}&&\Vol(\alpha')-t_\Sigma T(\gamma)^2=T(\gamma)^2\int_N\sigma\di\lambda_*,\\
(v)&\ \ \text{Fixed points:}&&\!\!\!\!\begin{array}[t]{l}\text{There is a map $\mathring N\cap\Fix(\varphi)\rightarrow\mathcal P_T(\alpha',\mathfrak h),\ q\mapsto \gamma_q$}\\
\text{such that $T(\gamma_q)=T(\gamma)(1+\sigma(q))$.}\end{array}\\
(vi)&\ \ \text{Zoll case:}&& \text{if $\alpha'$ is Zoll, then $\varphi= \id_N$}.
\end{align*}
\end{thm}
\begin{proof}
Let $C_0$ be the constant given by Lemma \ref{l:estomega} and let $\epsilon_{11}\leq \tfrac{1}{C_0}\epsilon_1$ be some positive real number, which will be determined in the course of the proof depending on $T$ and $\epsilon_{10}$. If $\alpha'$ is a contact form with $\Vert \di\alpha'-\di\alpha_*\Vert_{C^2}<\epsilon_{11}$, then Lemma \ref{l:estomega} yields a contact form $\alpha\in\mathcal B(C_0\epsilon_{11})$ with $\di\alpha=\di\alpha'$. Since $C_0\epsilon_{11}\leq \epsilon_1$, by Lemma \ref{l:neighbourhood1}, we have a period-preserving bijection $\mathcal P_T(\alpha',\mathfrak h)\rightarrow\mathcal P_T(\alpha,\mathfrak h)$ and by Proposition \ref{prp:gin}.(ii) the set $\mathcal P_T(\alpha,\mathfrak h)$ is compact and non-empty. 

We fix henceforth an element $\gamma\in\mathcal P_T(\alpha',\mathfrak h)$. If $\epsilon_2\in(0,\epsilon_1]$ is an auxiliary number, we can find a corresponding $\epsilon_3\in(0,\epsilon_0]$ according to Proposition \ref{prp:normal}, so that, if $\epsilon_{11}\leq\epsilon_3$ there exists a diffeomorphism $\Psi:\Sigma\to\Sigma$ such that $\alpha_{T(\gamma),\Psi}\in\mathcal B_*(\epsilon_2)$ and the map $\widetilde\gamma\mapsto \widetilde\gamma_{T(\gamma),\Psi}$ of Definition \ref{d:rescale} restricts to a bijection $\mathcal P_T(\alpha,\mathfrak h)\rightarrow\mathcal P_T(\alpha_{T(\gamma),\Psi},\mathfrak h)$. Thus, we get a bijection
\begin{equation}\label{e:bijectionppt}
\begin{array}{rcl}
\mathcal P_T(\alpha',\mathfrak h)&\longrightarrow&\mathcal P_T(\alpha_{c,\Psi},\mathfrak h),\\ \gamma'&\longmapsto&\gamma'_{T(\gamma),\Psi}
\end{array}\qquad T(\gamma')=T(\gamma)T(\gamma'_{T(\gamma),\Psi}).
\end{equation}
We choose now an auxiliary $\epsilon_4>0$ and get a corresponding $\epsilon_5\in(0,\epsilon_0]$ from Proposition \ref{p:nub}, so that if $\epsilon_{2}\leq \epsilon_5$, then there exist $\zeta:N\to N$ and $b:N\to\R$ associated with $\alpha_{T(\gamma),\Psi}\in\mathcal B_*(\epsilon_2)$ satisfying the properties contained therein. Finally, let $\epsilon_8>0$ be another auxiliary number and consider $\epsilon_9\in(0,\epsilon_0]$, the number given by Proposition \ref{p:return}, so that, if $\epsilon_2\leq\epsilon_9$, the statements contained therein hold for $\alpha_{T(\gamma),\Psi}$, the associated return time $\tau:N\to \R$ and return map $P:N\to N$.

Now we set
\[
\bullet\ \ \varphi:N\to N,\quad \varphi:=\zeta^{-1}\circ P\circ\zeta,\qquad \bullet\ \ \sigma:N\to\R,\quad \sigma:=\tau\circ\zeta-b\circ \varphi+b-1.
\]
First of all, we observe that by choosing $\epsilon_8$ and $\epsilon_4$ small enough, we obtain item (i). Then, we have $\varphi|_{\p N}=P|_{\p N}$ and $\sigma|_{\p N}=\tau|_{\p N}-1$, so that item (ii) follows from Proposition \ref{p:return}.(ii). As far as item (iii) is concerned, we compute 
\[
\varphi^*\lambda_*=\zeta^*P^*\lambda-\varphi^*\di b=\zeta^*(\lambda+\di\tau)-\di(b\circ\varphi)=\lambda_*+\di \big(b+\tau\circ\zeta-b\circ\varphi\big)=\lambda_*+\di\sigma.
\]
For item (iv), we recall from Lemma \ref{l:neighbourhood1}, Definition \ref{d:rescale} and Proposition \ref{p:return}.(iv) that
\[
\Vol(\alpha')=\Vol(\alpha)=T(\gamma)^2\Vol(\alpha_{T(\gamma),\Psi})=T(\gamma)^2\int_N\tau\di\lambda,
\]
and we will show that
\begin{equation}\label{e:intvolume}
\int_N\tau\di\lambda=\int_N\sigma\di\lambda_*+t_\Sigma.
\end{equation}
We can compute the integral of $\sigma\di\lambda_*$ as
\begin{align*}
\int_N\sigma\di\lambda_*=\int_N(\tau\circ\zeta)\di\lambda_*-\int_N(b\circ\varphi)\di\lambda_*+\int_N b\,\di\lambda_*-\int_N\di\lambda_*.
\end{align*}
We deal with the first summand. The map $\zeta$ preserves the orientation on $N$, as it is isotopic to the identity, and satisfies $\di\lambda_*=\zeta^*(\di \lambda)$. Hence,
\[
\int_N (\tau\circ \zeta)\di\lambda_*=\int_N (\tau\circ \zeta) \zeta^*(\di\lambda)=\int_N\zeta^*(\tau\,\di\lambda)=\int_N\tau\,\di\lambda.
\] 
The second and third summand cancel out. Indeed, as $\varphi$ preserves $\di\lambda_*$, we get
\[
\int_N (b\circ \varphi)\di\lambda_*=\int_N (b\circ \varphi) \varphi^*(\di\lambda_*)=\int_N\varphi^*(b\,\di\lambda_*)=\int_Nb\,\di\lambda_*.
\]
We deal with the last summand. By Stokes' Theorem, the fact that $\lambda_*|_{\p N}=-t_\Sigma\di\theta$ and that the induced orientation on $\p N$ is given by $-\di\theta$, we get
\[
\int_N\di\lambda_*=\int_{\p N}\lambda_*=-\int_0^1-t_\Sigma\di\theta=t_\Sigma.
\]
Plugging these last three identities in the computation above, we arrive at \eqref{e:intvolume}.

We move to item (v). We take $q\in \mathring N\cap\Fix(\varphi)$ and observe that $\zeta(q)\in\mathring N\cap\Fix(P)$. By Proposition \ref{p:return}.(v), there exists a periodic orbit $\gamma_q'\in\mathcal P(\alpha_{T(\gamma),\Psi},\mathfrak h)$ through $S(\zeta(q))$ with period $\tau(\zeta(q))\leq T$. We denote by $\gamma_q\in\mathcal P_T(\alpha',\mathfrak h)$ the orbit assigned to $\gamma'_q$ by the bijection given in \eqref{e:bijectionppt}, so that $T(\gamma_q)=T(\gamma) \tau(\zeta(q))$. Finally, we observe that
\[
\tau(\zeta(q))=1+\sigma(q)+b(\varphi(q))-b(q)=1+\sigma(q)+b(q)-b(q)=1+\sigma(q).
\]

The implication in item (vi) follows at once, since $\alpha'$ is Zoll if and only if $\alpha_{T(\gamma),\Psi}$ is Zoll by Lemma \ref{l:neighbourhood1}, and moreover, $P=\id_N$ if and only if $\varphi=\id_N$.
\end{proof}
\begin{cor}\label{c:neccond}
Suppose that we can choose $\epsilon_{10}$ in Theorem \ref{t:final3dim} so that, with the corresponding $\epsilon_{11}>0$, we have the following implications for a pair $(\varphi,\sigma)$ as above:
\begin{equation}\label{e:imply}
\begin{aligned}
\varphi\neq\id_N,\ \ \int_N\sigma\,\di\lambda_*\leq 0\quad\Longrightarrow\quad \exists\, q_-\in\mathring N\cap\Fix(\varphi),\ \ \sigma(q_-)<0,\\
\varphi\neq\id_N,\ \ \int_N\sigma\,\di\lambda_*\geq 0\quad\Longrightarrow\quad \exists\, q_+\in\mathring N\cap\Fix(\varphi),\ \ \sigma(q_+)>0.
\end{aligned}
\end{equation}
Then, Theorem \ref{t:main} holds taking $\mathcal U:=\big\{\Omega\ \text{exact two-form on }\Sigma\ \big|\ \Vert \Omega-\di\alpha_*\Vert_{C^2}<\epsilon_{11}\big\}$. 
\end{cor}
\begin{proof}
Let $\alpha'$ be a contact form such that $\di\alpha'\in\mathcal U$ as defined in the statement. If $\alpha'$ is Zoll, the conclusion follows from Proposition \ref{p:three}. Thus, we assume that $\alpha'$ is not Zoll and we want to prove that $\rho_{\mathrm{sys}}(\alpha',\mathfrak h)<\tfrac{1}{t_\Sigma}<\rho_{\mathrm{dia}}(\alpha',\mathfrak h)$. We first prove the inequality for the systolic ratio. Suppose by contradiction that $\rho_{\mathrm{sys}}(\alpha',\mathfrak h)\geq \tfrac{1}{t_\Sigma}$ and let $\gamma\in\mathcal P_T(\alpha',\mathfrak h)$ be such that
\begin{equation}\label{e:periodmin}
T(\gamma)=T_{\min}(\alpha',\mathfrak h),
\end{equation}
where $T_{\min}(\alpha',\mathfrak h)$ is the minimal period of prime periodic $\Phi^{\alpha'}$-orbits in the class $\mathfrak h$. The orbit $\gamma$ exists by Theorem \ref{t:final3dim}. Thus, the assumption $\rho_{\mathrm{sys}}(\alpha',\mathfrak h)\geq \tfrac{1}{t_\Sigma}$ implies
\begin{equation}\label{eq:vol-kT^2leq0}
\Vol(\alpha')-t_\Sigma T(\gamma)^2\leq 0.
\end{equation}
Theorem \ref{t:final3dim} assigns to $\gamma$ the pair $(\varphi,\sigma)$ with the properties listed therein. In particular, by Theorem \ref{t:final3dim}.(iv) and \eqref{eq:vol-kT^2leq0} above, we have that
\[
\int_N\sigma\di\lambda_*\leq0.
\] 
As $\alpha'$ is not Zoll, $\varphi\neq \id_N$ and we can use the first implication in \eqref{e:imply} to produce a point $q_-\in\mathring N\cap\Fix(\varphi)$ with $\sigma(q_-)<0$. By Theorem \ref{t:final3dim}.(v), this yields an element $\gamma_{q_-}\in\mathcal P_T(\alpha',\mathfrak h)$ with $T(\gamma_{q_-})<T(\gamma)$, which contradicts \eqref{e:periodmin}. This proves  $\rho_{\mathrm{sys}}(\alpha',\mathfrak h)< \tfrac{1}{t_\Sigma}$.
\medskip

The inequality with the diastolic ratio is analogously established. Suppose by contradiction that $\rho_{\mathrm{dia}}(\alpha',\mathfrak h)\leq \frac{1}{t_\Sigma}$. We take this time $\gamma\in\mathcal P_T(\alpha',\mathfrak h)$ to satisfy
\begin{equation}\label{e:periodmax}
T(\gamma)=T_{\max}(\alpha',\mathfrak h).
\end{equation}
If the pair $(\varphi,\sigma)$ is associated with $\gamma$, the assumption $\rho_{\mathrm{dia}}(\alpha',\mathfrak h)\leq \frac{1}{t_\Sigma}$ implies
\[
\int_N\sigma\di\lambda_*\geq0.
\]
The second implication in \eqref{e:imply} and Theorem \ref{t:final3dim}.(v) yield an orbit $\gamma_{q_+}\in\mathcal P_T(\alpha',\mathfrak h)$ with $T(\gamma_{q_+})> T(\gamma)$. This contradicts \eqref{e:periodmax} and proves $\rho_{\mathrm{dia}}(\alpha',\mathfrak h)> \frac{1}{t_\Sigma}$.
\end{proof}
In view of the last result, we only need to prove implications \eqref{e:imply} above to establish Theorem \ref{t:main}. This will be done in the next section.

\section{Surfaces with a symplectic form vanishing at the boundary}\label{sec:generating}
For $a>0$, we recall the notation for the annuli $\mathbb A=[0,a)\times S^1$, $\A'=[0,a/2)\times S^1$, where $S^1=\R/\Z$. As before if $\mathcal M$ is a manifold, we write $\mathring{\mathcal M}$ for the interior of $\mathcal M$.

In this section, $N$ will denote a connected oriented compact surface with one boundary component. We fix a collar neighbourhood of the boundary ${\mathfrak i}_\A:\A\to N$ with positively oriented coordinates $(r,\theta)\in\A$, where $r=0$ corresponds to $\p N$. Hence, we have the identification $S^1\cong \p N$ and the orientation induced by $N$ on $\p N$ is given by the one-form $-\di\theta$.

On $N$, we consider a one-form $\lambda$ such that $\di\lambda$ is a positive symplectic two-form on $\mathring N$ and
\begin{equation*}\label{e:deflambdaA}
\lambda_\A:={\mathfrak i}_\A^*\lambda=\big(-k+\tfrac12 r^2\big)\di\theta,
\end{equation*}
where, by Stokes' Theorem, $k=\int_N\di\lambda>0$. In particular, $\di\lambda_\A=r\di r\wedge\di \theta$ vanishes of order $1$ at $r=0$. The pair $(N,\lambda)$ is an instance of an \emph{ideal Liouville domain}, a notion due to Giroux (see \cite{Gir17}). 
\subsection{A neighbourhood theorem} 
In this subsection, we will develop a version of the Weinstein neighbourhood theorem for the diagonal
\[
\Delta_N\subset \big(N\times N,(-\di\lambda)\oplus\di\lambda\big).
\]
More precisely, we will consider the zero section
\[
\mathcal O_N\subset \big(\ta^*N,\di\lambda_{\mathrm{can}}\big)
\]
in the standard cotangent bundle of $N$ and look for an exact symplectic map $\mathcal W:\mathcal N\to \ta^*N$ from a neighbourhood $\mathcal N$ of $\Delta_N$ in $N\times N$, so that $\WW\circ {\mathfrak i}_{\Delta_N}={\mathfrak i}_{\OO_N}$, where
\[
{\mathfrak i}_{\Delta_N}:N\hookrightarrow N\times N,\quad {\mathfrak i}_{\Delta_N}(q)=(q,q),\qquad\quad {\mathfrak i}_{\OO_N}:N\hookrightarrow \ta^*N,\quad {\mathfrak i}_{\OO_N}(q)=(q,0)
\]
are the canonical inclusions of the diagonal and the zero section after the natural identifications of these sets with $N$. We start by giving an explicit construction of $\mathcal W$ on $\A\times \A$.

Let us endow the product $\A\times\A$ with coordinates $(r,\theta,R,\Theta)$, so that the diagonal is $\Delta_\A:=\{r=R,\ \theta=\Theta\}$. We make the identification $\ta^*\A=\A\times\R^2$ and let $(\rho,\vartheta,p_\rho,p_\vartheta)$ be the corresponding coordinates on $\ta^*\A$. We consider an open neighbourhood $\Y$ of $\Delta_\A$ defined as 
\begin{equation*}\label{def:y}
\Y:=\big\{(r,\theta,R,\Theta)\in\A\times\A\ \big|\ |\theta-\Theta|<\tfrac12\big\}
\end{equation*}
and define the auxiliary sets
\[
\Y':=\Y\cap(\A'\times\A'),\qquad \p \Y:=\Y\cap(\p \A\times\p \A).
\]
We have a well-defined difference function
\[
\Y\to(-\tfrac12,\tfrac12),\qquad (r,\theta,R,\Theta)\mapsto \theta-\Theta.
\]
We consider the map $\WW_\A:\Y\to \ta^*\A$ given in coordinates by
\begin{equation}\label{e:expwa}
\left\{
\begin{aligned}
\rho&=R,\\
\vartheta&=\theta,\\
p_\rho&=R(\theta-\Theta),\\
p_{\vartheta}&=\tfrac12(R^2-r^2),
\end{aligned}
\right.
\end{equation}
so that $\WW_\A\circ {\mathfrak i}_{\Delta_\A}={\mathfrak i}_{\OO_\A}$. The restriction $\WW_{\mathring{\A}}:=\WW_\A|_{\mathring\Y}:\mathring\Y\to \mathcal W_\A(\mathring\Y)$ is a diffeomorphism with inverse given by
\begin{equation}\label{e:expinvwa}
\left\{
\begin{aligned}
r&=\sqrt{\rho^2-2p_\vartheta},\\
\theta&=\vartheta,\\
R&=\rho,\\
\Theta&=\theta-\frac{p_\rho}{\rho}.
\end{aligned}
\right.
\end{equation}
We also consider the restriction $\WW_{\A'}:=\WW_\A|_{\Y'}:\Y'\to\ta^*\A'$. Its image has the following expression, which will be useful later on:
\begin{equation}\label{e:imagewa}
\mathcal W_{\A'}(\Y')=\Big\{(\rho,\vartheta,p_\rho,p_\vartheta)\in\ta^*\A'\ \Big|\ p_\rho\in\big(-\tfrac12\rho,\tfrac12\rho\big),\ p_\vartheta\in\big(\tfrac12\big(\rho^2-\tfrac{a^2}{4}\big),\tfrac12\rho^2\big]\Big\}.
\end{equation}
Finally, let us define the function
\begin{equation}\label{e:kaboundary}
K_\A:\Y\to \R,\qquad K_\A(r,\theta,R,\Theta):=(k-\tfrac12R^2)(\theta-\Theta),
\end{equation}
and set $K_{\A'}:=K_\A|_{\Y'}:\Y'\to\R$. There holds $K_\A|_{\Delta_\A}=0$ and
\begin{equation}\label{eq:psiK_A}
\WW_\A^*\lambda_\mathrm{can}=(-\lambda_\A)\oplus \lambda_\A-\di K_\A.
\end{equation}
Indeed, we have
\begin{align*}
\WW_\A^*\lambda_{\mathrm{can}}+\lambda_\A\oplus(-\lambda_\A)&=R(\theta-\Theta)\di R+\tfrac12(R^2-r^2)\di\theta+(-k+\tfrac12r^2)\di\theta-(-k+\tfrac12R^2)\di\Theta\\
&=(\theta-\Theta)\di(\tfrac12 R^2)+\tfrac12R^2\di(\theta-\Theta)-k\di(\theta-\Theta)\\
&=(\theta-\Theta)\di(-k+\tfrac12 R^2)+(-k+\tfrac12R^2)\di(\theta-\Theta)\\
&=-\di K_\A.
\end{align*}
Since, for all $q\in \A$, $(\lambda_{\mathrm{can}})_{(q,0)}=0$, we also deduce
\begin{equation}\label{e:dka}
\di_{(q,q)}K_\A=\big((-\lambda_\A)\oplus\lambda_\A\big)_{(q,q)}.
\end{equation}
Finally, if $\mathfrak i_{\p \Y}:\p\Y\to\Y$ is the natural inclusion, from \eqref{eq:psiK_A} we conclude that
\begin{equation}\label{e:dkapy}
\mathfrak i_{\p \Y}^*\big((-\lambda_\A)\oplus\lambda_\A\big)=\di\big(K_\A\circ\mathfrak i_{\p \Y}\big)
\end{equation}
Indeed, $(\WW_\A\circ\mathfrak i_{\p\Y})^*\lambda_{\mathrm{can}}=0$ from the explicit formula for $\mathcal W_\A$ given in \eqref{e:expwa} and the fact that both $r$ and $R$ vanish on $\p\Y$.

We can now state the neighbourhood theorem. The proof will be an adaptation of \cite[Theorem 3.33]{MS98} (with different sign convention). 
\begin{prp}\label{prop:neighborhood_map}
There exist an open neighbourhood $\NN\subset N\times N$ of the diagonal $\Delta_N$, a map $\WW:\NN\to\ta^*N$, and a function $K:\NN\to\R$ with the following properties.
\begin{enumerate}[(i)]
\item The set $\mathcal N$ contains $\Y'$. If we write \label{def:mathcalT}$\mathcal T:=\mathcal W(\mathcal N)$, then $\mathring{\mathcal T}\subset\ta^*{\mathring N}$ is an open neighbourhood of $\OO_{\mathring N}$ and the restriction $\WW|_{\mathring{\mathcal N}}:\mathring{\mathcal N}\to \mathring{\mathcal T}$ is a diffeomorphism.\\[-2ex]
\item $\WW^*\lambda_\mathrm{can}=(-\lambda)\oplus \lambda-\di K$.\\[-2ex]
\item $\WW\circ {\mathfrak i}_{\Delta_N}={\mathfrak i}_{\OO_N},\quad \WW|_{\Y'}=\WW_{\A'}$.\\[-2ex]
\item $K\circ {\mathfrak i}_{\Delta_N}=0,\quad K|_{\Y'}=K_{\A'}$.
\item If $\p\Y':=\Y'\cap(\p N\times\p N)$ and $\mathfrak i_{\p\Y'}:\p\Y'\to N\times N$ is the inclusion, then
\[
\di (K\circ\mathfrak i_{\p \Y'})=\mathfrak i_{\p\Y'}^*\big((-\lambda)\oplus\lambda\big).
\]
\end{enumerate}
\end{prp}

\begin{proof}
Let us denote by $(q,p)$ the points in $\ta^*\A\cong \A\times \R^2$, where $q=(\rho,\vartheta)$ and $p=(p_\rho,p_\vartheta)$. Let $g_\A$ and $g_{\ta^*\A}$ be the standard metrics on $\A$ and $\ta^*\A$:
\[
g_\A:=\di\rho^2+\di\vartheta^2,\qquad g_{\ta^*\A}:=\di\rho^2+\di\vartheta^2+\di p_\rho^2+\di p_\vartheta^2.
\]
Then, the metric $g_{\ta^*\A}$ is compatible with the canonical symplectic form $\di\lambda_{\op{can}}$. Namely,
\begin{equation*}
g_{\ta^*\A}=\di\lambda_{\op{can}}(J_{\ta^*\A}\,\cdot\,,\,\cdot\,),
\end{equation*}
where $J_{\ta^*\A}:\ta(\ta^*\A)\to\ta(\ta^*\A)$ is the standard complex structure given by
\[
J_{\ta^*\A}\p_\rho=\p_{p_\rho},\qquad J_{\ta^*\A}\p_\vartheta=\p_{p_\vartheta},\qquad J_{\ta^*\A}\p_{p_\rho}=-\p_{\rho},\qquad J_{\ta^*\A}\p_{p_\vartheta}=-\p_{\vartheta}.
\]
If $(q,0)\in\OO_\A$, then we have the horizontal and vertical embeddings
\[
\di_q{\mathfrak i}_{\OO_\A}:\ta_q\A\to \ta_{(q,0)}(\ta^*\A),\qquad \ta^*_q\A\to \ta_{(q,0)}(\ta^*\A),\ \ p\mapsto p^*,
\]
so that, if $\sharp:\ta^*_q\A\to \ta_q\A$ is the metric duality given by $g_\A$,  there holds
\[
p^*=J_{\ta^*\A}\cdot\di_q{\mathfrak i}_{\OO_\A}\cdot p^\sharp,\qquad \forall\, p\in\ta^*_q\A.
\] 
We now combine this formula with the fact that, for every $(q,p)\in \ta^*\A$, the ray $t\mapsto (q,tp)$, $t\in[0,1]$ is a geodesic for $g_{\ta^*\A}$ with initial velocity $p^*$. Thus, if $\exp^{\ta^*\A}$ denotes the exponential map of $g_{\ta^*\A}$, we arrive at
\begin{equation}\label{e:exploc}
(q,p)=\exp^{\ta^*\A}_{{\mathfrak i}_{\OO_{\A}}(q)}\Big(J_{\ta^*\A}\cdot \di_q {\mathfrak i}_{\OO_{\A}}\cdot p^\sharp\Big).
\end{equation}
We consider the pulled back objects $g_{\mathring \Y}:=\WW_{\mathring\A}^*g_{\ta^*\A}$ and $J_{\mathring\Y}:=\WW_{\mathring\A}^* J_{\ta^*\A}$. In particular, $\WW_{\mathring\A}$ is a local isometry between $g_{\mathring \Y}$ and $g_{\ta^*\A}$. Moreover, since ${\mathcal W}_{\mathring{\A}}^*(\di\lambda_{\op{can}})=((-\di\lambda_\A)\oplus\di\lambda_\A)$ by \eqref{eq:psiK_A}, we see that the structure $J_{\mathring \Y}$ is compatible with $(-\di\lambda)\oplus\di\lambda$, since $J_{\ta^*\A}$ is compatible with $\di\lambda_\mathrm{can}$ and $\lambda_\A=\mathfrak{i}_\A^*\lambda$. Namely,
\[
\big((-\di\lambda)\oplus\di\lambda\big)\big|_{\mathring \Y}=g_{\mathring\Y}(J_{\mathring\Y}\,\cdot\,,\,\cdot\,).
\]
Furthermore, using \eqref{e:exploc}, we can compute the pre-image of a point $(q,p)\in \mathcal W_{\A}(\mathring\Y)$ as
\begin{equation}\label{e:winverselocal}
\begin{aligned}
\WW^{-1}_{\mathring\A}(q,p)&=\WW^{-1}_{\mathring\A}\Big(\exp^{\ta^*\A}_{{\mathfrak i}_{\OO_{\A}}(q)}\big(J_{\ta^*\A}\cdot \di_q {\mathfrak i}_{\OO_{\A}}\cdot p^\sharp\big)\Big)\\
&=\exp^{\mathring\Y}_{{\mathfrak i}_{\Delta_\A}(q)}\Big(\di_{{\mathfrak i}_{\OO_\A}(q)}\WW^{-1}_{\mathring\A}\cdot J_{\ta^*\A} \cdot \di_q {\mathfrak i}_{\OO_\A}\cdot p^\sharp\Big)\\
&=\exp^{\mathring\Y}_{{\mathfrak i}_{\Delta_\A}(q)}\Big(J_{\mathring\Y}\cdot \di_{{\mathfrak i}_{\OO_\A}(q)}\WW^{-1}_{\mathring\A}\cdot \di_q {\mathfrak i}_{\OO_\A}\cdot p^\sharp)\Big)\\
&=\exp^{\mathring\Y}_{{\mathfrak i}_{\Delta_\A}(q)}\big(J_{\mathring\Y}\cdot \di_q {\mathfrak i}_{\Delta_\A}\cdot p^\sharp\big).
\end{aligned}
\end{equation}
The space of almost complex structures, which are compatible with the symplectic form $((-\di\lambda)\oplus \di\lambda)|_{\mathring N\times\mathring N}$, is contractible. Therefore, we can find an almost complex structure $J$ on $\mathring N\x \mathring N$, which is compatible with $((-\di\lambda)\oplus \di\lambda)|_{\mathring N\times\mathring N}$ and such that
\begin{equation}\label{e:Jmathringy}
J\big|_{\mathring\Y'}=J_{\mathring\Y}\big|_{\mathring\Y'}.
\end{equation}
We denote by $g$ the corresponding metric on $\mathring N\x \mathring N$, which satisfies
\begin{equation}\label{e:gmathringy}
g|_{\mathring\Y'}=g_{\mathring\Y}|_{\mathring\Y'}.
\end{equation}
We write $g_N:={\mathfrak i}_{\Delta_N}^*g$ for the restricted metric on $N$. We observe that $g_N|_{\mathring{\A}'}=g_\A|_{\mathring{\A}'}$, and therefore, we denote the metric duality given by $g_N$ also by $\sharp:\ta^*N\to\ta N$. Let us consider the set $\mathcal T_{1}$ made by all the points $(q,p)\in\ta^*\mathring N$ with the property that the $g$-geodesic starting at time $0$ from ${\mathfrak i}_{\Delta_N(q)}$ with direction $J\cdot\di_q{\mathfrak i}_{\Delta_N}\cdot p^\sharp$ is defined up to time $1$. We claim that
\begin{equation}\label{e:mathcalT}
\text{$\mathcal T_{1}$ is a fibre-wise star-shaped neighbourhood of $\OO_{\mathring N}$ and it contains $\mathcal W_{\A}(\mathring\Y')$}.
\end{equation}
The second assertion follows from equations \eqref{e:winverselocal} and \eqref{e:Jmathringy}, \eqref{e:gmathringy}. For the first one, we see from the homogeneity of the geodesic equation that $\mathcal T_{1}$ contains $\OO_{\mathring N}$, and it is fibre-wise star-shaped around $\OO_{\mathring N}$. Finally, since $\mathcal T_{1}\setminus \mathcal W_{\A}(\mathring\Y')$ is bounded away from $\p(\ta^*N)$, the set $\mathcal T_{1}$ is a neighbourhood of $\OO_{\mathring N}$. We define the map
\[
\Upsilon:\mathcal T_{1}\to N\times N,\qquad
\Upsilon(q,p):=\exp^{g}_{{\mathfrak i}_{\Delta_N}(q)}\Big(J\cdot \di_q {\mathfrak i}_{\Delta_N}\cdot p^\sharp\Big).
\]
It satisfies
\begin{equation}\label{e:upsilonrel}
\Upsilon|_{\mathcal W_{\A}(\mathring\Y')}=\WW_{\mathring \A'}^{-1},\qquad\Upsilon\circ  {\mathfrak i}_{\OO_N}={\mathfrak i}_{\Delta_N}.
\end{equation}
If $q\in\mathring N$, the differential of $\Upsilon$ at ${\mathfrak i}_{\OO_N}(q)$ in the direction $u=p^*+\di_{q}{\mathfrak i}_{\OO_N}\cdot v\in\ta_{{\mathfrak i}_{\OO_N}(q)}\ta^*N$, where $p\in\ta^*_qN$ and $v\in \ta_qN$, is given by
\[
\di_{{\mathfrak i}_{\OO_N}(q)}\Upsilon\cdot u= J\cdot \di_q{\mathfrak i}_{\Delta_N}\cdot p^\sharp+\di_q{\mathfrak i}_{\Delta_N}\cdot v.
\]
If we abbreviate $\Omega=(-\di\lambda)\oplus\di\lambda$, we claim that $(\Upsilon^*\Omega)_{{\mathfrak i}_{\OO_N}(q)}=(\di\lambda_\mathrm{can})_{{\mathfrak i}_{\OO_N}(q)}$, for all $q\in\mathring N$. For $u_1,u_2\in \ta_{{\mathfrak i}_{\OO_N}(q)}\ta^*N$, we compute
\begin{equation}\label{eq:equality_on_zero_section}
\begin{aligned}
\Upsilon^*\Omega(u_1,u_2)&=\Omega\big(J\cdot \di_q{\mathfrak i}_{\Delta_N}\cdot p_1^\sharp+\di_q{\mathfrak i}_{\Delta_N}\cdot v_1,J\cdot \di_q{\mathfrak i}_{\Delta_N}\cdot p_2^\sharp+\di_q{\mathfrak i}_{\Delta_N}\cdot v_2\big)\\
&=\Omega\big(J\cdot \di_q{\mathfrak i}_{\Delta_N}\cdot p_1^\sharp,\di_q{\mathfrak i}_{\Delta_N}\cdot v_2\big)-\Omega\big(J\cdot \di_q{\mathfrak i}_{\Delta_N}p_2^\sharp,\di_q{\mathfrak i}_{\Delta_N}\cdot v_1\big)\\
&=g\big(\di_q{\mathfrak i}_{\Delta_N}\cdot p_1^\sharp,\di_q{\mathfrak i}_{\Delta_N}\cdot v_2\big)-g\big(\di_q{\mathfrak i}_{\Delta_N}\cdot p_2^\sharp,\di_q{\mathfrak i}_{\Delta_N}\cdot v_1\big)\\
&=({\mathfrak i}_{\Delta_N}^*g)\big(p_1^\sharp, v_2\big)-({\mathfrak i}_{\Delta_N}^*g)\big(p_2^\sharp,v_1\big)\\
&=g_N\big(p_1^\sharp, v_2\big)-g_N\big(p_2^\sharp, v_1\big)\\
&=p_1(v_2)-p_2(v_1)\\
&=\di\lambda_\mathrm{can}(u_1,u_2),
\end{aligned}
\end{equation}
where in the second equality we used the fact that $\Delta_N$ is Lagrangian and that $J$ is a symplectic endomorphism. 

We move now the first steps in constructing the function $K:\mathcal N\to\R$. We abbreviate $\Lambda:=(\Upsilon^{-1})^*((-\lambda)\oplus\lambda)$. This is a one-form on $\mathcal T_{1}\subset \ta^*\mathring{N}$ and satisfies
\[
{\mathfrak i}_{\OO_N}^*\Lambda={\mathfrak i}_{\Delta_N}^*((-\lambda)\oplus\lambda)=-\lambda+\lambda=0.
\]
We consider any $K_{1}:\mathcal T_{1}\to N$ satisfying, for all $q\in\mathring N$,
\begin{equation}\label{eq:K_N}
K_{1}\circ {\mathfrak i}_{\OO_N}(q)=0,\qquad \di_{{\mathfrak i}_{\OO_N}(q)} K_{1}=\Lambda_{{\mathfrak i}_{\OO_N}(q)}.
\end{equation}
For example, we can set
\[
K_{1}(q,p):=\int_0^1\Lambda_{(q,tp)}(p^*)\di t.
\]
The first property in \eqref{eq:K_N} is immediate and it implies that
\[
\di_{{\mathfrak i}_{\OO_{N}}(q)}K_1\cdot\di_{q}{\mathfrak i}_{\OO_{N}}\cdot v=0=\Lambda_{{\mathfrak i}_{\OO_N}(q)}\big(\di_{q}{\mathfrak i}_{\OO_{N}}\cdot v\big).
\]
Thus, we just need to check the second property on vertical tangent vectors $p^*\in\ta_{(q,0)}(\ta^*N)$:
\begin{align*}
\di_{{\mathfrak i}_{\OO_N}(q)}K_1\cdot p^*=\lim_{s\to 0}\frac{K_1(q,sp)-K_1(q,0)}{s}=\lim_{s\to 0}\frac{1}{s}\int_0^1\Lambda_{(q,tsp)}(sp^*)\di t&=\lim_{s\to 0}\int_0^1\Lambda_{(q,tsp)}(p^*)\di t\\
&=\Lambda_{{\mathfrak i}_{\OO_N}(q)}(p^*).
\end{align*}

At this point, we transfer the attention on $N\times N$. First, we can shrink $\mathcal T_{1}$ in such a way that \eqref{e:mathcalT} still holds and that $\Upsilon$ is a diffeomorphism onto its image $\Upsilon(\mathcal T_{1})$. We define the open neighbourhood $\NN$ of $\Delta_N$ by
\[
\NN:=\Upsilon(\mathcal T_{1})\cup \Y'
\]
and the map $\WW_1:\NN\to\ta^*N$ obtained by gluing:
\begin{equation}\label{e:extensionw}
\WW_1|_{\Upsilon(\mathcal T_{1})}=\Upsilon^{-1},\qquad \WW_1|_{\Y'}=\WW_{\A'}.
\end{equation}
Such a map is well-defined because of \eqref{e:upsilonrel} and satisfies $\WW_1\circ {\mathfrak i}_{\Delta_N}={\mathfrak i}_{\OO_N}$. Let $\chi: \NN\to[0,1]$ be a cut-off function which is equal to $0$ in a neighbourhood of $\Y'$ and equal to $1$ on $\NN\setminus\Y$. We set
\[
K_\NN:\NN\to\R,\qquad K_{\mathcal N}:=\chi\cdot (K_1\circ \WW_1)+(1-\chi)\cdot K_\A.
\]
We readily see that
\begin{equation}\label{e:kappan}
K_\NN\circ {\mathfrak i}_{\Delta_N}=0,\qquad K_\NN|_{\Y'}=K_\A |_{\Y'}.
\end{equation}
Furthermore, for all $q\in \mathring N$, there holds
\begin{align*}
\di_{{\mathfrak i}_{\Delta_N}(q)} K_\NN&=\chi\cdot \WW_1^*(\di_{{\mathfrak i}_{\OO_N}(q)}K_{{1}})+(1-\chi)\cdot\di_{{\mathfrak i}_{\Delta_N}(q)} K_\A\\
&=\chi\cdot\WW_1^*(\Lambda_{{\mathfrak i}_{\OO_N}(q)})+(1-\chi)\cdot\big((-\lambda)\oplus\lambda\big)_{{\mathfrak i}_{\Delta_N}(q)}\\
&=\chi\cdot\big((-\lambda)\oplus\lambda\big)_{{\mathfrak i}_{\Delta_N}(q)}+(1-\chi)\cdot\big((-\lambda)\oplus\lambda\big)_{{\mathfrak i}_{\Delta_N}(q)},\\
&=\big((-\lambda)\oplus\lambda\big)_{{\mathfrak i}_{\Delta_N}(q)},
\end{align*}
where we used $K_1\circ \WW_1\circ\mathfrak i_{\Delta_N}(q)=0=K_\A\circ\mathfrak i_{\Delta_N}(q)$ in the first equality. while
the second equality followed from \eqref{e:dka} and \eqref{eq:K_N}. Since $(\lambda_{\mathrm{can}})_{{\mathfrak i}_{\OO_N}(q)}=0$, we deduce
\begin{equation}\label{e:exactat0}
(\WW_1^*\lambda_{\mathrm{can}})_{{\mathfrak i}_{\Delta_N}(q)}=\big((-\lambda)\oplus\lambda\big)_{{\mathfrak i}_{\Delta_N}(q)}-\di_{{\mathfrak i}_{\Delta_N}(q)}K_\NN.
\end{equation}
The rest of the proof follows Moser's argument. We set 
\[
\Lambda_t:=t\big(\WW_1^*\lambda_\mathrm{can}+\di K_\NN\big)+(1-t)\big((-\lambda)\oplus\lambda\big),\quad t\in[0,1].
\] 
By \eqref{eq:psiK_A}, \eqref{e:extensionw}, and \eqref{e:kappan}, we have
\begin{equation}\label{e:lambday}
\Lambda_t= (-\lambda)\oplus\lambda\ \ \text{on}\  \Y'.
\end{equation}
Moreover, for all $q\in \mathring N$, by \eqref{eq:equality_on_zero_section} and \eqref{e:exactat0}, we have
\begin{equation}\label{e:atzero}
(\di\Lambda_t)_{{\mathfrak i}_{\Delta_N}(q)}=\big((-\di\lambda)\oplus\di\lambda\big)_{{\mathfrak i}_{\Delta_N}(q)},\qquad (\Lambda_t)_{{\mathfrak i}_{\Delta_N}(q)}=\big((-\lambda)\oplus\lambda\big)_{{\mathfrak i}_{\Delta_N}(q)}.
\end{equation}
In particular $\di\Lambda_t$ is non-degenerate on $\Delta_{\mathring{N}}$. Therefore, up to shrinking the neighbourhood away from $\Y'$, we can assume that $\di\Lambda_t$ is non-degenerate on $\mathring\NN$. Let $X_t$ be a time-dependent vector field and $L_t$ a time-dependent function on $\mathring\NN$ defined by 
\[
\iota_{X_t}\di\Lambda_t=-\frac{\di \Lambda_t}{\di t},\qquad L_t:=-\int_0^t\Lambda_{t'}(X_{t'})\circ \Phi_{t'}\,\di t',
\]
where $\Phi_{t}$ is the flow of $X_t$. By \eqref{e:lambday}, we see that $X_t$ and $L_t$ vanish on $\mathring\Y'$ and we can extend them trivially to the whole $\NN$. Relations \eqref{e:atzero} imply that $X_t$ and $L_t$ vanish on $\Delta_{\mathring N}$. In particular, $\Phi_t$ is the identity map on $\Delta_{N}$, and up to shrinking the neighbourhood $\NN$ away from $\Y'$, we can suppose that $\Phi_t$ is defined up to time $1$. For all $t\in[0,1]$ we have
\[
\frac{\di}{\di t}\Big(\Phi_t^*\Lambda_t+\di L_t\Big)=\Phi_t^*\Big(\iota_{X_t}\di\Lambda_t+\di\big(\Lambda_t(X_t)\big)+\frac{\di\Lambda_t}{\di t}\Big)+\di\Big(\frac{\di L_t}{\di t}\Big)=0.
\]
Together with $\Phi_0^*\Lambda_0+\di L_0=\Lambda_0$, this implies $\Phi_1^*\Lambda_1+\di L_1=\Lambda_0$. Hence,
\[
\Phi_1^*\WW_1^*\lambda_\mathrm{can}=(-\lambda)\oplus\lambda-\di\big(L_1+K_{\mathcal N}\circ\Phi_1\big),
\]
and properties (i) and (ii) in the statement follow with  
\[
\WW:= \WW_1\circ \Phi_1,\qquad K:=L_1+K_{\mathcal N}\circ \Phi_1.
\]
Properties (iii) and (iv) hold as well, since they are satisfied by $\WW_1$ and $K_{\mathcal N}$ and we have shown that $\Phi_1|_{\Delta_N}=\id$, $\Phi_1|_{\Y'}=\id$ and $L_1|_{\Delta_N}=0$, $L_1|_{\Y'}=0$. Property (v) follows from (iv) and equation \eqref{e:dkapy}.
\end{proof}

\subsection{Exact diffeomorphisms $C^1$-close to the identity}

Let $\mathbb E$ denote the set of all exact diffeomorphisms $\varphi:N\to N$, namely $\varphi^*\lambda-\lambda$ is an exact one-form. We endow from now on $\mathbb E$ with the uniform $C^1$-topology, whose associated distance function we denote by $\dist_{C^1}$. For  $\epsilon>0$, we consider the open ball around $\id_N$ of radius $\epsilon$ 
\[
\mathbb E(\epsilon):=\big\{\varphi\in\mathbb E\ \big|\ \dist_{C^1}(\varphi,\id_N)<\epsilon\big\}.
\]

If $\varphi\in\mathbb E$, we write \label{def:gammaphi}$\Gamma_\varphi:N\to N\times N$ for its graph $\Gamma_\varphi(q)=(q,\varphi(q))$, and we have $\Gamma_\varphi(\p N)\subset \p N\times\p N$. There is \label{def:epsilon*}$\epsilon_*>0$ such that all $\varphi\in\mathbb E(\epsilon_*)$ enjoy the following properties:
\begin{enumerate}[(a)]
\item $\Gamma_\varphi(N)\subset\NN$.
\item If $\pi_N:\ta^*N\to N$ is the foot-point projection, then the map 
\begin{equation*}\label{e:defnu}
\nu_\varphi:N\to N,\qquad \nu_\varphi:=\pi_N\circ \WW\circ \Gamma_{\varphi}
\end{equation*}
is a diffeomorphism. Indeed, $\nu_\varphi$ is $C^1$-close to $\id_N$, if the same is true for $\varphi$. Henceforth, we write $\nu$ instead of $\nu_\varphi$ when the map $\varphi$ is clear from the context.
\item We have the inclusions
\[
\varphi\big(\A''\big)\subset\A',\qquad \nu^{-1}(\A'')\subset\A',
\]
where $\A'':=[0,a/4)\times S^1$.
\end{enumerate}

If $\varphi\in\mathbb E(\epsilon_*)$, then we can write its restriction to $\A''$ as
\begin{equation*}
\varphi(r,\theta)=\big(R_\varphi(r,\theta),\Theta_\varphi(r,\theta)\big).
\end{equation*}
By \eqref{e:expwa}, the restriction of $\nu$ to $\A''$ reads
\begin{equation}\label{e:nurtheta}
\nu(r,\theta)=\big(R_\varphi(r,\theta),\theta\big),
\end{equation}
which implies that
\begin{equation*}
\nu_\varphi|_{\p N}=\id_{\p N}.
\end{equation*}
Let $\mathfrak i_{\p N}:\p N\to N$ be the inclusion and observe that $\Gamma_\varphi\circ\mathfrak i_{\p N}$ takes values in $\p \Y'$. Therefore, taking the pull-back by $\Gamma_\varphi\circ\mathfrak i_{\p N}$ in Proposition \ref{prop:neighborhood_map}.(v), we get
\[
\di\big(K\circ\Gamma_\varphi\circ\mathfrak i_{\p N}\big)=\mathfrak i_{\p N}^*\big(\varphi^*\lambda-\lambda\big).
\]
With this relation we can single out a special primitive of $\varphi^*\lambda-\lambda$ called the action of $\varphi\in\mathbb E(\epsilon_*)$. It is the unique $C^1$-function $\sigma:N\to \R$ such that
\begin{equation}\label{e:normalization_sigma}
(i)\ \ \varphi^*\lambda-\lambda=\di \sigma,\qquad (ii)\ \ \sigma|_{\p N}=K\circ \Gamma_\varphi|_{\p N}.
\end{equation}

\begin{rmk}
We observe that the normalisation of $\sigma$ at the boundary coincides with the one considered in \eqref{e:lambdaxy} and Theorem \ref{t:final3dim}. Indeed, we have the explicit formulas $\lambda=-k\di\theta$ on $\p N$ and $K(0,\theta,0,\Theta)=k(\theta-\Theta)$ on $\p \Y'$, and for all $\theta_0\in S^1\cong \p N$, there holds
\[
K\circ \Gamma_\varphi(\theta_0)=-k\big(\Theta_\varphi(\theta_0)-\theta_0\big)=-k\int_{\theta_0}^{\varphi(\theta_0)}\di\theta=\int_{\theta_0}^{\varphi(\theta_0)}\lambda.
\]
\end{rmk}
We describe the tangent space of $\mathbb E(\epsilon_*)$. To this purpose we introduce a space of functions.
\begin{dfn}\label{d:normc2+}
We write $\mathbb V$ for the vector space of all smooth functions $f:N\to\R$ such that both $f$ and $\di f$ vanish at $\p N$. We endow this space with the pre-Banach norm $\Vert\cdot\Vert_{\mathbb V}$ defined by
\begin{equation*}
\Vert f\Vert_{\mathbb V}:=\Vert f\Vert_{C^2}+\Vert \tfrac1r\di f|_{\A}\Vert_{C^1},\qquad \forall\, f\in\mathbb V.
\end{equation*}
Choosing the restriction to a smaller annulus in the second term above yields an equivalent norm on $\mathbb V$. For all $\delta>0$, we denote by $\mathbb V(\delta)$ the open ball of radius $\delta$ in $(\mathbb V,\Vert\cdot\Vert_{\mathbb V})$. 
\end{dfn}
Let $\varphi$ denote some element in $\mathbb E(\epsilon_*)$ with action $\sigma$. First, we take any differentiable path $t\mapsto \varphi_t$ with values in $\mathbb E(\epsilon_* )$ such that $\varphi=\varphi_0$, and write $t\mapsto \sigma_t$ for the corresponding path of actions with $\sigma=\sigma_0$. Let $X_t$ be the $C^1$-vector field on $N$ uniquely defined by
\begin{equation}\label{e:varphiX}
\frac{\di\varphi_t}{\di t}=X_t\circ\varphi_t.
\end{equation}
The associated Hamiltonian function is defined by
\begin{equation}\label{e:defhamiltonian2}
H_t:N\to \R,\qquad H_t:= \frac{\di\sigma_t}{\di t}\circ\varphi_t^{-1}-\lambda(X_t).
\end{equation}
Differentiating $\varphi_t^*\lambda=\lambda+\di\sigma_t$ with respect to $t$, we get
\begin{equation*}
\iota_{X_t}\di\lambda=\di H_t.
\end{equation*}
From this last equation and the fact that $\di\lambda=R\di R\wedge \di \Theta$ vanishes at $\p N$, we see that $\di H_t$ vanishes at $\p N$. Hence, if we write $X_t=X_t^R\p_R+X_t^\Theta\p_\Theta$ on the annulus $\A$, we find
\begin{equation}\label{e:defhamiltonianbound}
X_t^R=\tfrac1R \p_\Theta H_t,\qquad X_t^\Theta=-\tfrac1R\p_R H_t.
\end{equation}
We also observe that $H_t=0$ at the boundary $\p N$ since $\frac{\di\sigma_t}{\di t}=\lambda(X_t)\circ\varphi_t$ there. Indeed, from \eqref{e:normalization_sigma} and Proposition \ref{prop:neighborhood_map}.(v), we compute at $\p N$
\[
\frac{\di\sigma_t}{\di t}=\di_{\Gamma_{\varphi_t}}K\cdot(0\oplus X_t)=\big((-\lambda)\oplus\lambda\big)(0\oplus X_t)\big|_{\Gamma_{\varphi_t}}=\lambda(X_t)\circ\varphi_t.
\]
Therefore, we see that $H_t$ belongs to $\mathbb V$ and $\Vert H_t\Vert_{\mathbb V}$ is equivalent to $\Vert X_t\Vert_{C^1}$.

Conversely, let $H\in\mathbb V$ and take any path $t\mapsto H_t$ with values in $\mathbb V$ and such that $H_0=H$. We claim that there is a uniquely defined path $t\mapsto X_t$ of $C^1$-vector fields with $\iota_{X_t}\di\lambda=\di H_t$. The vector fields are well defined away from $\p N$, since $\di\lambda$ is symplectic there. On $\A$, instead, they are well defined because of \eqref{e:defhamiltonianbound}. Let $t\mapsto \varphi_t$ be the path of diffeomorphisms obtained integrating $X_t$ with the condition $\varphi_0=\varphi$. Differentiating with respect to $t$, we get 
\[
\frac{\di}{\di t}\big(\varphi_t^*\lambda\big)=\varphi_t^*\big(\iota_{X_t}\di\lambda+\di(\lambda(X_t))\big)=\di\big((H_t+\lambda(X_t))\circ\varphi_t\big)
\]
so that all the maps $\varphi_t$ are exact with some action $\sigma_t$. Relation \eqref{e:defhamiltonian2} is also satisfied since $H_t$ and $\frac{\di\sigma_t}{\di t}\circ\varphi_t^{-1}-\lambda(X_t)$ have the same differential and both vanish at $\p N$. We sum up the previous discussion in a lemma.
\begin{lem}\label{l:tev}
There is an isomorphism between the pre-Banach spaces 
\[
\big(\ta_\varphi \mathbb E(\epsilon_*),\Vert \cdot\Vert_{C^1}\big)\longrightarrow(\mathbb V,\Vert\cdot\Vert_{\mathbb V})
\] 
given by the map $X_0\mapsto H_0$, where $X_0$ and $H_0$ are defined in \eqref{e:varphiX} and \eqref{e:defhamiltonian2}.\qed
\end{lem}

\subsection{Generating functions}
In this subsection, we describe how to build the correspondence between $C^1$-small exact diffeomorphisms and generating functions in our setting. For a classical treatment, we refer the reader to \cite[Chapter 9]{MS98}. Let $\varphi$ be an exact diffeomorphism in $\mathbb E(\epsilon_*)$. There exists a one-form $\eta:N\to \ta^*N$ such that 
\[
\WW\circ\Gamma_\varphi=\eta\circ\nu.
\]
Since $\lambda_{\mathrm{can}}$ has the tautological property $\eta^*\lambda_{\mathrm{can}}=\eta$, we have
\begin{equation}\label{e:nueta}
\begin{aligned}
\nu^*\eta=\nu^*\eta^*\lambda_{\mathrm{can}}=\Gamma_\varphi^*\WW^*\lambda_{\mathrm{can}}=\Gamma_\varphi^*\big((-\lambda)\oplus\lambda-\di K\big)&=\varphi^*\lambda-\lambda-\di(K\circ\Gamma_\varphi)\\
&=\di(\sigma-K\circ\Gamma_\varphi).
\end{aligned}
\end{equation}
If we denote the generating function of $\varphi\in\mathbb E(\epsilon_*)$ by
\begin{equation}\label{e:G}
G_\varphi:N\rightarrow \R,\qquad G_\varphi:=(\sigma-K\circ\Gamma_\varphi)\circ\nu_\varphi^{-1},
\end{equation}
we have the equality
\begin{equation}\label{eq:dG}
\WW\circ\Gamma_\varphi=\di G_\varphi\circ\nu.
\end{equation}
Henceforth, we will simply write $G$ instead of $G_\varphi$ when the map $\varphi$ is clear from the context. 

We write the restriction of $\nu^{-1}$ to $\A''$ as $\nu^{-1}(\rho,\vartheta)=\big(r_\varphi(\rho,\vartheta),\vartheta\big)$, so that, for every $\theta=\vartheta$, the functions $R_\varphi(\cdot,\theta)$ and $r_\varphi(\cdot,\vartheta)$ are inverse of each other. Moreover, since $r_\varphi(0,\vartheta)=0$, by Taylor's theorem with integral remainder, there exists a function $s_\varphi:\A''\to \R$ such that
\begin{equation*}
r_\varphi=\rho (1+s_\varphi).
\end{equation*}
By \eqref{e:nurtheta}, \eqref{eq:dG} and \eqref{e:expwa}, we have
\begin{equation}\label{e:gatbound}
\left\{\begin{aligned}
\p_\rho G(\rho,\vartheta)&=\rho\big(\vartheta-\Theta_\varphi(r_\varphi(\rho,\vartheta),\vartheta)\big),\\[1ex] 
\p_\vartheta G(\rho,\vartheta)&=\tfrac12\big(\rho^2-r_\varphi^2(\rho,\vartheta)\big)=-\rho^2\big(\tfrac12s_\varphi^2(\rho,\vartheta)+s_\varphi(\rho,\vartheta)\big).
\end{aligned}\right.
\end{equation}

\begin{prp}\label{p:fixg}
If $G:N\to \R$ is the generating function of $\varphi\in\mathbb E(\epsilon_*)$, there holds
\begin{align*}
\mathring N\cap\Fix(\varphi)=\mathring N\cap\Crit G.
\end{align*}
Moreover, if $z\in \mathring N\cap\Fix(\varphi)$, then $\nu(z)=z$ and $\sigma(z)=G(z)$.
\end{prp}

\begin{proof}
Let $z$ be a point in $\mathring N$. We suppose first that $\varphi(z)=z$. Then, $\Gamma_\varphi(z)\in\Delta_N$, and by (iii) in Proposition \ref{prop:neighborhood_map}, we have $\WW(\Gamma_\varphi(z))=\mathfrak i_{\OO_N}(z)$, which implies that $\nu(z)=z$ and $\di_z G=0$. Moreover, by \eqref{e:G} and Proposition \ref{prop:neighborhood_map}.(iv), we have
\[
G(z)=\sigma(\nu^{-1}(z))-K\big(\Gamma_\varphi(\nu^{-1}(z))\big)=\sigma(z)-K\circ\mathfrak i_{\Delta_N}(z)=\sigma(z).
\]
Conversely, suppose that $z$ is a critical point $G$. Then, by \eqref{eq:dG}
\[
(z,z)=\mathfrak i_{\Delta_N}(z)=\WW^{-1}(\di G(z))=\Gamma_{\varphi}(\nu^{-1}(z))=(\nu^{-1}(z),\varphi(\nu^{-1}(z))),
\]
which implies $\nu^{-1}(z)=z$, and hence, $\varphi(z)=z$. 
\end{proof}

\begin{lem}\label{lem:normalization_G}
The generating function $G$ belongs to $\mathbb V$. Moreover, there holds
\begin{equation*}
\p^2_{\rho\rho}G|_{\p N}=\id_{\p N}-\Theta_\varphi\circ\mathfrak i_{\p N}.
\end{equation*}
In particular, for every $z\in\p N$, we have
\begin{equation*}
\varphi(z)=z\quad\iff\quad \p^2_{\rho\rho}G(z)=0,\quad\iff\quad\sigma(z)=0.
\end{equation*}
\end{lem}
\begin{proof}
The vanishing of $G$ at the boundary follows from (ii) in \eqref{e:normalization_sigma}. To prove the vanishing of the differential of $G$ at the boundary, we just substitute $\rho=0$ in \eqref{e:gatbound}. Moreover, dividing the first equation in \eqref{e:gatbound} by $\rho$ and taking the limit for $\rho$ going to $0$, we obtain the formula for $\p^2_{\rho\rho}G$, which also implies the first equivalence above. The second one follows from \eqref{e:kaboundary} and \eqref{e:normalization_sigma}.
\end{proof}
By the previous lemma we have a map
\begin{equation*}\label{e:imageG}
\mathcal G:\mathbb E(\epsilon_*)\to\mathbb V,\qquad \mathcal G(\varphi)=G_\varphi,
\end{equation*}
whose properties we will study. To this aim, we need a definition and two lemmas about functions on $\A$.
\begin{dfn}\label{d:c0+}
Fix a positive integer $m$. Let us denote by $\mathbb F$ the space of all smooth functions $\hat f:\A\to \R^m$ and by $\Vert\cdot\Vert_{\mathbb F}$ the norm on $\mathbb F$ defined by
\begin{equation*}
\Vert \hat f\Vert_{\mathbb F}:=\Vert \hat f\Vert_{C^0}+\Vert r\di \hat f\Vert_{C^0},\qquad \forall\, \hat f\in\mathbb F.
\end{equation*}
Let $\mathbb F_0\subset \mathbb F$ be the subspace of those functions $f:\A\to\R^m$ such that $f(0,\theta)=0$, for all $\theta\in S^1$. In this case, there exists a unique $\hat f\in\mathbb F$ such that
\begin{equation*}
f(r,\theta)=r\hat f(r,\theta),\qquad \forall\,(r,\theta)\in\A. 
\end{equation*}
\end{dfn} 
\begin{lem}\label{l:divi}
The following two statements hold.
\begin{enumerate}[(i)]
\item The map $(\mathbb F_0,\Vert\cdot\Vert_{C^1})\to(\mathbb F,\Vert\cdot\Vert_{\mathbb F})$, $f\mapsto \hat f$ is an isomorphism of pre-Banach spaces. 
\item Let $U$ be an open set of $\R^m$, and let $A:U\to \R^m$ be a $C^2$-function with $\Vert A\Vert_{C^2}<\infty$. If $\mathbb F_U$ is the set of all functions $\hat f\in\mathbb F$ such that the image of $\hat f$ is a relatively compact subset of $U$, then the following map is continuous:
\begin{equation*}
(\mathbb F_U,\Vert\cdot\Vert_{\mathbb F})\to (\mathbb F,\Vert \cdot\Vert_{\mathbb F}),\qquad \hat f\mapsto A\circ \hat f.
\end{equation*}
\end{enumerate}
\end{lem}
\begin{proof}
By Taylor's theorem with integral remainder, the function $\hat f$ is defined as
\begin{equation}\label{eq:hat_f}
\hat f(r,\theta)=\int_0^1\p_r f(ur,\theta)\,\di u.
\end{equation}
Moreover, differentiating the identity $f=r\hat f$, we deduce that
\begin{equation}\label{eq:dH}
\di f=r\di \hat f+\hat f\di r.
\end{equation}
We see from \eqref{eq:hat_f} that the $C^0$-norm of $\hat f$ is controlled by the $C^1$-norm of $f$. Consequently, from \eqref{eq:dH}, we conclude that the $C^0$-norm of $r\di \hat f=\di f-\hat f\di r$ is also controlled by the $C^1$-norm of $f$. On the other hand, we deduce from \eqref{eq:dH} that the $C^0$-norm of $\di f$ is controlled by the $C^0$-norm of $r\di \hat f$ and $\hat f$. As $f$ vanishes at $r=0$, the $C^0$-norm of $f$ is controlled, as well.

Finally, we consider a map $A:U\to \R^m$ as in the statement. Let $\hat f_0\in\mathbb F_U$ be fixed and $\hat f\in \mathbb F_U$ such that $\hat f_0+r(\hat f-\hat f_0)\in\mathbb F_U$, for all $r\in[0,1]$. This happens if $\hat f$ is $C^0$-close to $\hat f_0$ since the images of $\hat f_0$ and $\hat f$ are relatively compact in $U$, by assumption. Then, we can estimate with the help of the mean value theorem:
\begin{align*}
\Vert A\circ \hat f-A\circ \hat f_0\Vert_{C^0}&\leq \Vert A\Vert_{C^1}\Vert \hat f-\hat f_0\Vert_{C^0};\\[1ex]
\big\Vert r\di(A\circ \hat f-A\circ \hat f_0)\big\Vert_{C^0}&=\big\Vert (r\di_{\hat f}A\cdot \di \hat f-r\di_{\hat f}A\cdot\di \hat f_0)+(r\di_{\hat f}A\cdot \di \hat f_0-r\di_{\hat f_0}A\cdot \di \hat f_0)\big\Vert_{C^0}\\
&\leq \big\Vert \di_{\hat f} A\cdot r\di (\hat f-\hat f_0)\big\Vert_{C^0}+\big\Vert (\di_{\hat f} A-\di_{\hat f_0}A)\cdot r\di \hat f_0\big\Vert_{C^0}\\
&\leq \Vert A\Vert_{C^1}\Vert r\di (\hat f-\hat f_0)\Vert_{C^0}+\Vert A\Vert_{C^2}\Vert \hat f-\hat f_0\Vert_{C^0}\Vert r\di \hat f_0\Vert_{C^0},
\end{align*}
from which the continuity of the map $\hat f\mapsto A\circ \hat f$ at $\hat f_0$ follows.
\end{proof}

\begin{lem}\label{lem:functions_g_h}
Let $f:\A\to \R$ be a function such that, for all $\vartheta\in S^1$, we have $f(0,\vartheta)=0$, $\di_{(0,\vartheta)}f=0$. Then, there exist functions $f_\rho, f_\vartheta:\A\to \R$ such that 
\[
\p_\rho f=\rho f_\rho,\quad \p_\vartheta f=\rho^2 f_\vartheta.
\]
Moreover, there exists a constant $C>0$ (independent of $f$) such that
\begin{equation*}
\frac1C\left\Vert\frac{1}{\rho}\di f\right\Vert_{C^1}\leq \Vert f_\rho\Vert_{C^1}+\Vert f_\vartheta\Vert_{\mathbb F}\leq C\left\Vert\frac{1}{\rho}\di f\right\Vert_{C^1}.
\end{equation*}
\end{lem}
\begin{proof}
By Taylor's theorem with integral remainder, for all $(\rho,\vartheta)\in\A$, we can write
\[
f(\rho,\vartheta)=\rho^2 \hat{\hat f}(\rho,\vartheta),
\]
for a function $\hat{\hat f}:\A\to \R$, so that $f_\rho:=2\hat{\hat f}+\rho\p_\rho \hat{\hat f}$, $f_\vartheta:=\p_{\vartheta}\hat{\hat f}$ yield the desired functions. In order to prove the equivalence of the norms, we observe that $\tfrac{1}{\rho}\di f=f_\rho\di \rho+\rho f_\vartheta\di\vartheta$. Thus, $\tfrac1\rho\di f$ is $C^1$-small if and only if $f_\rho$ is $C^1$-small and $\rho f_\vartheta$ is $C^1$-small. The conclusion now follows from Lemma \ref{l:divi}.(i).
\end{proof}
\begin{prp}\label{p:contg}
The map $\mathcal G:\mathbb E(\epsilon_*)\to \mathbb V$ is continuous from the $C^1$-topology to the topology induced by $\Vert\cdot\Vert_{\mathbb V}$.
\end{prp}
\begin{proof}
Since we can write $\di G=\WW\circ\Gamma_\varphi\circ\nu^{-1}$, we readily see that the map $\mathcal G$ is continuous from the $C^1$-topology to the topology induced by the $C^2$-norm. The lemma follows if we can establish the continuity from the $C^1$-topology to the topology induced by the semi-norm $\Vert\tfrac 1\rho \di (\,\cdot\,)|_{\A''}\Vert_{C^1}$. If $\pi_{S^1}:\A''\to S^1$ is the standard projection, then, using equations \eqref{e:gatbound}, this amounts to showing that the map
\[
\varphi\mapsto\pi_{S^1}-\Theta_\varphi\circ\nu_\varphi^{-1}
\]
is continuous from the $C^1$-topology to the $C^1$-topology, and further employing Lemma \ref{l:divi}.(i), that the map
\[
\varphi\mapsto -\tfrac12 s_\varphi^2-s_\varphi
\]
is continuous from the $C^1$-topology to the $\|\cdot\|_{\mathbb F}$-topology. The former map is continuous since $(f_1,f_2)\mapsto f_1\circ f_2$ is continuous from the product $C^1$-topology into the $C^1$-topology and
\[
\varphi\mapsto \Theta_\varphi,\qquad \varphi\mapsto \nu_\varphi^{-1}=(r_\varphi,\pi_{S^1})
\]
are continuous in the $C^1$-topology. The latter map is continuous since
\begin{enumerate}[(a)]
\item the map $\varphi\mapsto s_\varphi=\tfrac1\rho (r_\varphi-\rho)$ is continuous from the $C^1$-topology to the $\|\cdot\|_{\mathbb F}$-topology by Lemma \ref{l:divi}.(i);
\item the map $\hat f\mapsto A\circ \hat f$ with $A(x)=-\tfrac12 x^2-x$ is continuous from the $\|\cdot\|_{\mathbb F}$-topology to the $\|\cdot\|_{\mathbb F}$-topology by Lemma \ref{l:divi}.(ii). 
\end{enumerate}
Putting everything together, we have shown that $\mathcal G$ is continuous.
\end{proof}

It is well known that the map $\mathcal W$ translates the standard Hamiltonian-Jacobi equation for exact Lagrangian graphs in $\ta^*N$ to the Hamilton-Jacobi equation for $C^1$-small exact diffeomorphisms. Namely, for every differentiable path $t\mapsto\varphi_t\subset \mathbb E(\epsilon_*)$ with $\nu_t:=\nu_{\varphi_t}$ and generated by some $t\mapsto H_t$, the corresponding path $t\mapsto G_t:=\mathcal G(\varphi_t)\subset\mathbb V$ has a smooth pointwise derivative $t\mapsto \tfrac{\di G_t}{\di t}\subset\mathbb V$ which satisfies the Hamilton-Jacobi equation:
\begin{equation}\label{e:hj}
\frac{\di G_t}{\di t}\circ \nu_t=H_t\circ\varphi_t.
\end{equation}
By continuity it is enough to show \eqref{e:hj} on the interior $\mathring N$ where $d\lambda$ is symplectic. We define the extended Hamiltonian $\widetilde H_t:\mathring N\times \mathring N\to\R$ by $\widetilde H_t(q,Q):=H_t(Q)$. It generates $\tilde\varphi_t:=\id\times \varphi_t$ on $\mathring N\times \mathring N$ which is Hamiltonian with respect to $(-\di\lambda)\oplus\di\lambda$. By definition, $\Gamma_{\varphi_t}=\tilde{\varphi}_t\circ \mathfrak i_{\Delta_{\mathring N}}$. Thus, if we write $\psi_t:\mathcal T\to \ta^*{\mathring N}$ for the Hamiltonian diffeomorphisms defined on a neighbourhood of $\mathcal O_{\mathring N}\subset\ta^*{\mathring N}$ generated by $\widetilde H_t\circ \mathcal W^{-1}$, we get $\di G_t\circ \nu_t=\psi_t\circ\mathfrak i_{\mathcal O_{\mathring N}}$ by \eqref{eq:dG}. Hence, $G_t$ solves the classical Hamilton-Jacobi equation with respect to $\widetilde H_t\circ \mathcal W^{-1}$ \cite[Section 46D]{Arn}, i.e.~$\frac{\di G_t}{\di t}=\widetilde H_t\circ \mathcal W^{-1}(\di G_t)$. 
From the definition of $\widetilde H_t$ and identity \eqref{eq:dG}, we obtain \eqref{e:hj}.

\begin{rmk}
If we endow $\mathbb E(\epsilon_*)$ with the $C^2$-topology (instead of the coarser $C^1$-topology), then the map $\mathcal G$ becomes of class $C^1$, and for all $\varphi\in \mathbb E(\epsilon_*)$ and $H\in\mathbb V\cong\ta_\varphi\mathbb E(\epsilon_*)$, we can rephrase the equation in the statement of the proposition as
\[
\di_\varphi\mathcal G\cdot H=H\circ(\varphi\circ\nu^{-1}).
\]
\end{rmk}
\begin{prp}\label{prp:bijectivity_of_Xi}
There are $\delta_{*},\epsilon_{**}>0$ and a continuous map $\mathcal E:\mathbb V(\delta_{*})\to \mathbb E(\epsilon_*)$ such that
\begin{enumerate}[(i)]
	\item we have the inclusion $\mathcal G(\mathbb E(\epsilon_{**}))\subset\mathbb V(\delta_{*})$;
	\item the map $\mathcal E$ is the inverse of $\mathcal G$, namely,
	\begin{equation*}
\bullet\ \ \mathcal G\big(\mathcal E(G)\big)=G,\quad \forall\, G\in \mathbb V(\delta_{*}),\qquad \bullet\ \ \mathcal E(\mathcal G(\varphi))=\varphi,\quad\forall\, \varphi\in\mathbb E(\epsilon_{**}).	
	\end{equation*}
\end{enumerate}
\end{prp}
\begin{proof}
Let $\delta_*$ be a positive number. We first show that if $G\in\mathbb V(\delta_*)$, then $\di G$ takes values into $\mathcal T=\WW(\NN)$, provided $\delta_*$ is small enough. Since $\mathcal T$ is a neighbourhood of the zero section away from the boundary of $N$, we see that $\di G(N\setminus \A'')$ is contained in $\mathcal T$ if $\delta_*$ is small. On the other hand, since $\mathcal T\supset \mathcal W(\mathbb Y')$ from Proposition \ref{prop:neighborhood_map}.(i), we just need to show that $\di G(\A'')\subset\mathcal W(\mathbb Y')\cap (\ta^*\A'')$. Recall from \eqref{e:imagewa} the description
\[
\mathcal W(\mathbb Y')\cap (\ta^*\A'')=\Big\{(\rho,\vartheta,p_\rho,p_\vartheta)\in\ta^*\A''\ \ \Big|\ \ p_\rho\in \big(-\tfrac12\rho,\tfrac12\rho\big),\ \ p_\vartheta\in\big(\tfrac12\big(\rho^2-\tfrac{a^2}{4}\big),\tfrac12\rho^2\big]\Bigg\},
\]
so that the implication
\[
0\leq\rho<\tfrac a4\quad\Longrightarrow\quad \tfrac12\big(\rho^2-\tfrac{a^2}{4}\big)>-\tfrac32\rho^2,
\]
yields the implication
\[
(\rho,\vartheta,p_\rho,p_\vartheta)\in\mathcal W(\mathbb Y')\cap (\ta^*\A'')\quad\Longrightarrow\quad p_\vartheta\in(-\tfrac32\rho^2,\tfrac12\rho^2].
\]
By Lemma \ref{lem:functions_g_h}, we have the expressions $\p_\rho G=\rho G_\rho$ and $\p_\vartheta G=\rho^2 G_\vartheta$. Therefore, in order to have $\di G(\A'')\subset \mathcal W(\mathbb Y') \cap(\ta^*\A'')$, we just need $\Vert G_\rho\Vert_{C^0(\A'')}<\tfrac12$ and $\Vert G_\vartheta\Vert_{C^0(\A'')}<\tfrac12$, which are true if $\delta_*$ is small, thanks to the inequality in Lemma \ref{lem:functions_g_h} and the definition of $\Vert\cdot\Vert_{\mathbb V}$.

Since $\di G(\mathring N)\subset \mathring{\mathcal T}$, we can consider the map
\begin{equation}\label{e:checkmu}
\mathring\mu:\mathring N\to \mathring N,\quad\mathring\mu:=\pi_1\circ\WW^{-1}\circ \di G|_{\mathring N}, 
\end{equation}
where $\pi_1:N\x N\to N$ is the projection on the first factor. On the annulus $\A''$, we consider, furthermore, the map
\[
\mu_{\A''}:\A''\to\A',\qquad \mu_{\A''}(\rho,\vartheta)=\big(\rho\sqrt{1-2G_\vartheta(\rho,\vartheta)},\vartheta\big).
\]
Thanks to \eqref{e:expinvwa}, $\mathring \mu$ and $\mu_{\A''}$ glue together and yield a map $\mu_G:N\to N$. We claim that $G\mapsto \mu_G$ is continuous from the topology induced by $\Vert\cdot\Vert_{\mathbb V}$ to the $C^1$-topology. We argue separately for $\mathring \mu|_{N\setminus\A''}$ and $\mu_{\A''}$. For the former map, the continuity is clear from the expression \eqref{e:checkmu} and the fact that $\Vert G\Vert_{C^2}\leq \Vert G\Vert_{\mathbb V}$. For the latter map, the continuity is clear in the second factor, and we only have to deal with the continuity of $G\mapsto \rho\sqrt{1-2G_\vartheta}$. By Lemma \ref{l:divi}, this happens if and only if $G\mapsto \sqrt{1-2G_\vartheta}$ is continuous from the $\Vert\cdot\Vert_{\mathbb V}$-topology to the $\|\cdot\|_{\mathbb F}$-topology. The latter map is the composition of $G\mapsto G_\vartheta$ with $f\mapsto A\circ f$, where $A:(-\tfrac12,+\tfrac12)\to(0,\infty)$ is defined by $A(x)=\sqrt{1-2x}$. The map $G\mapsto G_\varphi$ is continuous from the $\Vert\cdot\Vert_{\mathbb V}$-topology to the $\|\cdot\|_{\mathbb F}$-topology by Lemma \ref{lem:functions_g_h}. The map $f\mapsto A\circ f$ is continuous from the $\|\cdot\|_{\mathbb F}$-topology to the $\|\cdot\|_{\mathbb F}$-topology by Lemma \ref{l:divi}.(ii). The claim is established.

Thus, taking $\delta_*$ small enough, we can assume that $\mu_G:N\to N$ is so $C^1$-close to the identity that is a diffeomorphism and we write $\nu_G:N\to N$ for its inverse, which satisfies
\begin{equation*}
\nu_G(r,\theta)=\big(R_G(r,\theta),\theta\big),\quad\forall\, (r,\theta)\in \A'',
\end{equation*}
for some function $R_G:\A''\to [0,a/2)$. The map $G\mapsto \nu_G$ is continuous in the $C^1$-topology.

We now construct a diffeomorphism $\varphi_G:N\to N$. Let $\pi_2:N\x N\to N$ be the projection on the second factor and set
\begin{equation}\label{e:ginj2}
\mathring\varphi:\mathring N\to\mathring N,\quad \mathring \varphi:=\pi_2\circ\WW^{-1}\circ \di G\circ\nu_G|_{\mathring N}.
\end{equation}
On the annulus $\A''$, we set
\[
\varphi_{\A''}:\A''\to\A',\qquad \varphi_{\A''}(r,\theta)=\big(R_G(r,\theta),\theta-G_\rho(R_G(r,\theta),\theta)\big).
\]
Thanks to \eqref{e:expinvwa}, the maps $\mathring \varphi$ and $\varphi_{\A''}$ glue together to yield $\varphi_G:N\to N$. We claim that $\varphi$ is exact. Indeed, from \eqref{e:checkmu} and \eqref{e:ginj2}, we get $\WW\circ\Gamma_{\mathring\varphi}=\di G\circ\mathring\nu$. Since $\nu_G$ and $\varphi_G$ are continuous up to the boundary, we deduce
\begin{equation}\label{e:varphiG}
\WW\circ\Gamma_{\varphi_G}=\di G\circ\nu_G.
\end{equation}
Repeating the computation as in \eqref{e:nueta}, it follows that $\varphi_G$ is exact with action
\begin{equation}\label{e:sigmavarphiG}
\sigma_{\varphi_G}:=G\circ\nu_G+K\circ\Gamma_{\varphi_G}.
\end{equation}
Therefore, we have a map $\mathcal E:\mathbb V(\delta_*)\to\mathbb E$ defined by $\mathcal E(G)=\varphi_G$. We claim that this map is continuous. As before, we argue separately for $\mathring \varphi|_{N\setminus\A''}$ and $\varphi_{\A''}$. For the former map, the continuity follows since we have a control on the $C^2$-norm of $G$. For the latter map, it follows from the continuity of $G\mapsto R_G$ from the $\Vert\cdot\Vert_{\mathbb V}$-topology to the $C^1$-topology, which we have already established, the continuity of $G\mapsto G_\rho$ from the $\Vert\cdot\Vert_{\mathbb V}$-topology to the $C^1$-topology, which follows from Lemma \ref{lem:functions_g_h}, and the continuity of $(f_0,f_1)\mapsto f_0\circ f_1$ from the product $C^1$-topology into the $C^1$-topology. The claim is established. In particular, up to shrinking $\delta_*$, we can assume that $\mathcal E(\mathbb V(\delta_*))\subset\mathbb E(\epsilon_*)$. On the other hand, the existence of $\epsilon_{**}>0$ with the property that $\mathcal G(\mathbb E(\epsilon_{**}))\subset \mathbb V(\delta_*)$ is a consequence of the continuity of $\mathcal G$.

Next, we verify that $\mathcal G(\varphi_G)=G$. First, recalling that $\pi_N:\ta^*N\to N$, we see that\vspace{-5pt}
\[
\nu_{\varphi_G}\stackrel{\eqref{eq:dG}}{=}\pi_N\circ(\WW\circ\Gamma_{\varphi_G})\stackrel{\eqref{e:varphiG}}{=}\pi_N\circ(\di G\circ\nu_G)=(\pi_N\circ\di G)\circ\nu_G=\nu_G.
\]
Therefore, comparing \eqref{e:sigmavarphiG} with \eqref{e:G}, we get $\mathcal G(\varphi_G)=G_{\varphi_G}=G$.

Finally, let $\varphi\in\mathbb E(\epsilon_{**})$. We show that $\varphi=\mathcal E(G_\varphi)$. First, we get
\[
\nu_{\varphi}^{-1}|_{\mathring N}\stackrel{\eqref{eq:dG}}{=}\pi_1\circ \WW^{-1}\circ\di G_{\varphi}|_{\mathring N}\stackrel{\eqref{e:varphiG}}{=}\pi_1\circ \Gamma_{\varphi_{G_\varphi}}\circ \nu^{-1}_{G_\varphi}|_{\mathring N}=\nu_{G_\varphi}^{-1}|_{\mathring N}.
\]
By continuity, this implies $\nu_\varphi=\nu_{G_\varphi}$, and we arrive at\vspace{-5pt}
\[
\varphi|_{\mathring N}\stackrel{\eqref{eq:dG}}{=}\pi_2\circ \WW^{-1}\circ\di G_\varphi\circ\nu_\varphi|_{\mathring N}=\pi_2\circ \WW^{-1}\circ\di G_\varphi\circ\nu_{G_\varphi}|_{\mathring N}\stackrel{\eqref{e:varphiG}}{=}\pi_2\circ\Gamma_{\varphi_{G_\varphi}}|_{\mathring N}=\varphi_{G_\varphi}|_{\mathring N}.
\]
By continuity again, $\varphi=\varphi_{G_\varphi}=\mathcal E(G_\varphi)$ as required, and the proof is completed.
\end{proof}

\subsection{Quasi-autonomous diffeomorphisms}\label{sec:calabi}

In this subsection, we complete the proof of Theorem \ref{t:main} using arguments inspired by \cite[Remark 2.8]{ABHS15}.
We begin with the following well-known lemma whose proof can be found in \cite[Lemma 10.27]{MS98} and \cite[Proposition 2.6 \& 2.7]{ABHS15}.
\begin{lem}\label{l:sigmaqa}
Let $\varphi\in\mathbb E(\epsilon_*)$ be an exact diffeomorphism and let $\sigma:N\to\R$ denote its action. Suppose that there exists a differentiable path $t\mapsto\varphi_t$ in $\mathbb E(\epsilon_*)$ with $\varphi_0=\id_N$ and $\varphi_1=\varphi$. We write by $t\mapsto H_t\in\mathbb V$ the Hamiltonian associated with the path. There holds
\begin{equation*}
\sigma(q)=\int_0^1\Big[H_t+\lambda(X_t)\Big](\varphi_t(q))\,\di t=\int_0^1\big(t\mapsto \varphi_t(q)\big)^*\lambda+\int_0^1 H_t(\varphi_t(q))\,\di t, \quad\forall\, q\in N.
\end{equation*}
As a consequence, we have
\begin{equation*}
\int_N\sigma\,\di\lambda=2\int_0^1\Big(\int_NH_t\di \lambda\Big)\di t.\tag*{\qed}
\end{equation*}
\end{lem}

We recall that, according to \cite{BP94}, a Hamiltonian path $t\mapsto H_t\in\mathbb V$, parametrised in some interval $I$, is called quasi-autonomous if there exist a minimiser $q_{\min}\in N$ and a maximiser $q_{\max}\in N$ independent of time, i.e. 
\[
\min_{N} H_t=H_t(q_{\min}),\qquad\max_{N} H_t=H_t(q_{\max}), \qquad\forall\, t\in I.
\]
A diffeomorphism $\varphi\in\mathbb E(\epsilon_*)$ is called quasi-autonomous, if there exists a differentiable path $t\mapsto\varphi_t\in\mathbb E(\epsilon_*)$ parametrised in $[0,1]$ with $\varphi_0=\id_N$, $\varphi_1=\varphi$, whose associated Hamiltonian $t\mapsto H_t\in\mathbb V$ is quasi-autonomous.

\begin{lem}\label{l:qa}
Let $\varphi\in\mathbb E(\epsilon_*)$ be quasi-autonomous with associated Hamiltonian $t\mapsto H_t$. The following implications hold:
\[
\begin{aligned}
\exists\, t_-\in[0,1],\ H_{t_-}(q_{\min})<0,&\qquad\Longrightarrow\qquad q_{\min}\in \Fix(\varphi)\cap \mathring N,\ \sigma(q_{\min})<0,\\
\exists\, t_+\in[0,1],\ H_{t_+}(q_{\max})<0,&\qquad\Longrightarrow\qquad q_{\max}\in \Fix(\varphi)\cap \mathring N,\ \sigma(q_{\max})<0.\\
\end{aligned}
\]
\end{lem}
\begin{proof}
We show only the first implication. Since $H_{t_-}(q_{\min})<0$ and $H_{t_-}|_{\p N}=0$, we deduce that $q_{\min}\in \mathring N$. Moreover, since $q_{\min}$ minimises $H_t$ for all $t\in[0,1]$, we see that $\di_{q_{\min}}H_t=0$. Since $\di\lambda$ is symplectic on $\mathring N$, by $\iota_{X_t}\di\lambda=\di H_t$, we conclude that $X_t(q_{\min})=0$, which implies that $\varphi_t(q_{\min})=q_{\min}$. We estimate the action of $q_{\min}$ using Lemma \ref{l:sigmaqa} and remembering that, for all $t\in[0,1]$, there holds $H_t(q_{\min})\leq 0$, since $H_t$ vanishes on the boundary:
\[
\sigma(q_{\min})=\int_0^1\big[H_t+\lambda(X_t)\big](\varphi_t(q_{\min}))\di t=\int_0^1 H_t(q_{\min})\di t<0.\qedhere
\]
\end{proof}
\vspace{3pt}
\begin{prp}\label{p:qa}
Every $\varphi\in\mathbb E(\epsilon_{**})$ is quasi-autonomous.
\end{prp}
\begin{proof}
By Proposition \ref{prp:bijectivity_of_Xi}, the generating function $G$ of $\varphi$ belongs to $\mathbb V(\delta_*)$. Thus, for all $t\in[0,1]$, the function $tG$ belongs to $\mathbb V(\delta_*)$, and again by Proposition \ref{prp:bijectivity_of_Xi}, we can consider the path $t\mapsto \varphi_t:=\mathcal E(tG)\in\mathbb E(\epsilon_*)$. Let $t\mapsto H_t$ be the associated Hamiltonian. By \eqref{e:hj}, we deduce
\begin{equation}\label{e:hj2}
G=\frac{\di}{\di t}(tG)=H_t\circ(\varphi_t\circ\nu_t^{-1}),\quad\forall\, t\in[0,1],
\end{equation}
which implies
\begin{equation}\label{e:hj3}
\min H_t=\min G,\qquad \max H_t=\max G,\qquad\forall\, t\in[0,1].
\end{equation}
Let $q_{\min}$ and $q_{\max}$ be the minimiser and the maximiser of $G$, respectively. We claim that
\begin{equation}\label{e:hj4}
G(q_{\min})=H_t(q_{\min}),\qquad G(q_{\max})=H_t(q_{\max}),\qquad\forall\, t\in[0,1].
\end{equation}
We give only the argument for $q_{\min}$. If $q_{\min}\in \p N$, we have $G(q_{\min})=0=H_t(q_{\min})$, as $G$ and $H_t$ belong to $\mathbb V$. If $q_{\min}\in\mathring{N}$, then $q_{\min}\in\Crit G$. We deduce that $\varphi_t(q_{\min})=q_{\min}=\nu_t(q_{\min})$, as $\varphi_t$ and $\nu_t$ act as the identity on $\mathring N\cap\Crit (tG)\supset\mathring N\cap\Crit G$ by Proposition \ref{p:fixg}. The equality $G(q_{\min})=H_t(q_{\min})$ follows then from \eqref{e:hj2}. Now that the claim is established, relations \eqref{e:hj3} and \eqref{e:hj4} imply that $t\mapsto H_t$ is quasi-autonomous.
\end{proof}
We are now ready to prove implications \eqref{e:imply} in Corollary \ref{c:neccond}, which are the last missing piece to establish the Main Theorem \ref{t:main}.
\begin{cor}\label{cor:implications}
Let $\varphi\in\mathbb E(\epsilon_{**})$ be an exact diffeomorphism with action $\sigma:N\to \R$. If $\varphi\neq\id_N$, the following implications hold:
\begin{align*}
\bullet&\quad \int_N\sigma\,\di\lambda\leq0\quad\Longrightarrow\quad \exists\, q_{-}\in\Fix(\varphi)\cap\mathring N\;\text{ with }\; \sigma(q_-)<0,\\
\bullet&\quad \int_N\sigma\,\di\lambda\geq0\quad\Longrightarrow\quad \exists\, q_{+}\in\Fix(\varphi)\cap\mathring N\;\text{ with }\; \sigma(q_+)<0.
\end{align*}
\end{cor}
\begin{proof}
The implications follow with $q_-=q_{\min}$, $q_+=q_{\max}$. We show only the former, the latter being analogous. Suppose that the integral of $\sigma$ is non-positive. By Proposition \ref{p:qa}, $\varphi$ is quasi-autonomous, namely, there exists a quasi-autonomous $t\mapsto H_t$ generating $t\mapsto \varphi_t$ with $\varphi_0=\id_N$ and $\varphi_1=\varphi$. By Lemma \ref{l:qa}, the corollary is established, if we show that there exists $t_-\in[0,1]$ such that $H_{t_-}(q_{\min})<0$. Indeed, assume by contradiction that $H_t(q_{\min})\geq 0$, for all $t\in[0,1]$. This means that $H_t\geq 0$. Furthermore, as $\varphi\neq\id_N$, there exists $(s,w)\in[0,1]\times N$ with $H_{s}(w)>0$, which, by Lemma \ref{l:sigmaqa}, implies
\[
0<\int_0^1\Big(\int_N H_t\Big)\di t=\frac12\int_N\sigma\di\lambda.
\]
From this contradiction we conclude the existence of a $t_-$ as above.  
\end{proof}

\bibliographystyle{amsalpha}
\bibliography{systolic_bib}
\end{document}